\documentclass[reqno,11pt]{amsart}
\usepackage{amsmath,amssymb,amsthm,bm,amsfonts,mathrsfs,graphicx}
\usepackage{breqn}
\usepackage[mathcal]{euscript}
\usepackage[breaklinks=true]{hyperref}
\usepackage{cleveref,enumitem}
\usepackage{aliascnt}
\usepackage{macrostxd}
\usepackage[english]{babel}
\usepackage{setspace}
\usepackage{braket}
\usepackage{latexsym,bm}
\usepackage{amscd}
\usepackage{epstopdf}
\usepackage[all]{xypic}
\usepackage[titletoc]{appendix}
\usepackage{xfrac,xcolor}
\usepackage[normalem]{ulem}

    \makeatletter
    \def\step{
    \@ifnextchar[ \@step{\@noitemargtrue\@step[\@itemlabel]}}
    \def\@step[#1]{
        \item[#1]\textit{}\hspace*{\dimexpr-\labelwidth-\labelsep}
        }
    \makeatother
\setenumerate{fullwidth,itemindent=\parindent,
    listparindent=\parindent,itemsep=0ex,
    partopsep=0pt,parsep=0ex}

\begin{document}
\title[]{Entire sign-changing solutions to the  fractional critical Schr{\"o}dinger equation}

\author[]{Xingdong Tang}
 \address{\hskip-1.15em Xingdong Tang
 \hfill\newline
 School of Mathematics and Statistics,
 \hfill\newline
 Nanjing University of Information Science and Technology,
 \hfill\newline
 Nanjing, 210044, People's Republic of China.}
 \email{txd@nuist.edu.cn}

\author[]{Guixiang Xu}
 \address{\hskip-1.15em Guixiang Xu
 	 	\hfill\newline Laboratory of Mathematics and Complex Systems (Ministry of Education),
 	\hfill\newline School of Mathematical Sciences,
 	\hfill\newline Beijing Normal University,
 	\hfill\newline Beijing, 100875, People's Republic of China.}
 \email{guixiang@bnu.edu.cn}

\author[]{Chunyan Zhang}
 \address{\hskip-1.15em Chunyan Zhang \hfill\newline School of Mathematical Sciences,
 	\hfill\newline Nanjing Normal University,
 	\hfill\newline Nanjing, 210046, People's Republic of China.}
 \email{chunyan7755@126.com }

\author[]{Jihui Zhang}
 \address{\hskip-1.15em Jihui Zhang \hfill\newline School of Mathematical Sciences,
 	\hfill\newline Nanjing Normal University,
 	\hfill\newline Nanjing, 210046, People's Republic of China.}
 \email{zhangjihui@njnu.edu.cn }

\subjclass[2010]{Primary: 35A15, 35J91; Secondary: 35R11}
\begin{abstract}
We consider the fractional  critical Schr{\"o}dinger equation (FCSE)
\begin{align*}
  \slaplace{u}-\abs{u}^{2^{\ast}_{s}-2}u=0,
\end{align*}
where $u \in \dot H^s( \R^N)$,  $N\geq 2$,  $0<s<1$  and $2^{\ast}_{s}=\frac{2N}{N-2s}$.
By virtue of
    the mini-max theory
   and the concentration compactness principle
    with  the  equivariant group action,
we obtain the new type of non-radial, sign-changing solutions of (FCSE) in the energy space $\dot H^s(\R^N)$. The key component is that we use the equivariant group to partion $\dot H^s(\R^N)$ into several connected components, then combine  the concentration compactness argument  to show the compactness property of Palais-Smale sequences in each component and obtain many solutions of (FCSE) in $\dot H^s(\R^N)$. Both the solutions and the argument here are different from those by
Garrido, Musso in \cite{GM2016pjm}
 and by
Abreu, Barbosa and Ramirez in \cite{ABR2019arxiv}.

\end{abstract}
\keywords{Action of the Equivariant Group; Concentration-compactness Principle;  Fractional Critical Schr\"odinger Equation; Sign-changing Solution.}
\maketitle
\section{Introduction}
This paper is concerned with the existence of sign-changing solutions to
the following fractional critical Schr\"odinger equation
\begin{align}\label{nls}
\begin{cases}
  \slaplace{u}-\abs{u}^{2^{\ast}_{s}-2}u=0, & \mbox{in } \R^N, \\
  u\in \dot{H}^{s}\sts{\R^N}, & \
\end{cases}
\end{align}
where $N\geq2$, $0<s<1$, $2^{\ast}_{s}=\frac{2N}{N-2s}$, $\slaplace$ denotes the usual fractional Laplace operator
and
$\dot{H}^{s}\sts{\R^N}$ denotes the homogenous Sobolev space of real-valued functions whose energy
associated to $ \slaplace $ is finite, i.e.
\begin{equation}\label{hs}
  \dot{H}^{s}\sts{\R^N} =
  \setcdt{ u\in \Scal'\sts{\R^N}
  }{
  \norm{ u }_{\dot{H}^s}<+\infty
    },
\end{equation}
with
$$
    \norm{ u }_{\dot{H}^s}^2 = \int_{\R^N}\abs{ \xi}^2\abs{\sts{\Fcal{u}}\sts{\xi} }^2\d\xi,
$$
where
 $\Fcal{u}$ denotes the Fourier transform of $u$:
\begin{equation*}
  \Fcal{u}\sts{\xi}=\frac{1}{\sts{2\pi}^{ \frac{N}{2} }}\int_{\R^N}u\sts{x}\e^{-\i x\xi}\d x.
\end{equation*}

Fractional Schr\"odinger equations \eqref{nls} arise as models in the fractional quantum mechanics, including
path integral over the L{\'e}vy flights paths (see for instance \cite{Laskin2000pre,Laskin2018,Laskin2000}),
and
as Euler-Lagrange equations for the Hardy-Littlewood-Sobolev inequalities (e.g., see \cite{CT2004jmaa,FL2010cvpde,Lieb1983am}).


The problem about the positive solutions to \eqref{nls} has attracted lots of attention. On the one hand, the existence of positive solutions to \eqref{nls} is  related to
the existence of extremizers to the Hardy-Littlewood-Sobolev inequalities. Lieb considered the following Hardy-Littlewood-Sobolev inequality in \cite{Lieb1983am}
\begin{equation}
\label{liebhls}
  \norm{u}_{L^{\frac{2N}{N-2s}}}^2\leq S\sts{N,s}\norm{ u }_{\dot{H}^{s}\sts{\R^N}}^2,
\end{equation}
and obtained that $\omega_{\mu,\lambda,x_0}$ is the extremizer to \eqref{liebhls} if and only if
\begin{equation}
\label{liebextremals}
  \omega_{\mu,\lambda,x_0}\sts{x}=\frac{ \mu }{\lambda^{\frac{N-2s}{2} }}
  \sts{ \frac{1}{ 1+\frac{ \abs{x-x_0}^2 }{ \lambda^2 } } }^{ \frac{N-2s}{2} },
  \quad
  \mu\neq 0,
  \quad
  \lambda>0
  \text{~~and~~}
  x_0\in\R^N,
\end{equation}
by the layer cake representation technique. Note that \eqref{liebextremals} also solves \eqref{nls} by taking suitable choices of $\mu$. We can refer to \cite{CL1990jfa, CT2004jmaa, FL2010cvpde}, et.al.. for more references.
On the other hand, up to the symmetries of \eqref{nls},
Chen, Li and Ou made use of the moving plane method to show that  \eqref{liebextremals} are the only positive (negative) solutions
to \eqref{nls}  in $L^{ 2N/(N-2s)}_{\text{loc}}\sts{\R^N}$ in \cite{CLO2006cpam}.  Moreover,   D{\'{a}}vila,  del Pino and  Sire  obtained the nondegeneracy of
the extremizer \eqref{liebextremals} for the Hardy-Littlewood-Sobolev inequality \eqref{liebhls} in \cite{DPS2013tams}.

Sign-changing solutions to \eqref{nls}, in the case $s=1$, has been intensively studied in \cite{Clapp2016jde, CS2020cpde, PMP2013adsns, ding1986cmp, HV1994jfa, MMW2019jfa}, et al. As far as the authors known,
there are two different ways to study sign-changing solutions of \eqref{nls}.
On the one hand, Ding  obtained infinitely many sign-changing solutions by making use of variational methods restricted to the space of group invariant functions in \cite{ding1986cmp}. Clapp showed the multiplicity
of sign-changing solutions by making use of minimax argument restricted to the space of group equivariant functions in  \cite{Clapp2016jde}.  We can also refer to \cite{CS2020cpde} for the application in critical Lane-Emden systems.  On the other hand,
del Pino, Musso, Pacard and Pistoia constructed sign-changing solutions
by the Lyapunov–Schmidt reduction argument in \cite{PMPP2011jde, PMP2013adsns, MM2020arxiv}.
Recently, Medina and Musso also constructed
some kind of sign-changing solutions with maximal rank in \cite{MM2020arxiv}.

Our main result  is the following.
\begin{theo}
\label{thm:main}
Let $N=4n+m$ with $n\geq1$ and $m\in\{0,1,2,3\}$. Then for any $0<s<1$, the problem \eqref{nls} has at least $n$ non-radial sign-changing solutions.
\end{theo}

Both the result and the argument in this paper are different from those
in \cite{ABR2019arxiv, GM2016pjm}.  Abreu, Barbosa and Ramirez obtained infinitely many sign-changing solutions of \eqref{nls} by the Ljusternik-Schnirelman type mini-max method and group invariant technique in   \cite{ABR2019arxiv}.
Garrido and Musso  constructed the sign-changing solutions of \eqref{nls}
by the  Lyapunov–Schmidt reduction argument in \cite{GM2016pjm}.  The key idea here is that we use the equivariant group to partion $\dot H^s(\R^N)$ into several connected components, then combine  the concentration compactness argument  to show the compactness property of Palais-Smale sequences in each component and obtain many solutions of \eqref{nls} in $\dot H^s(\R^N)$, where the compactness property of the Palais-Smale sequences is nontrivial for the fractional case $0<s<1$, please see \Cref{sec:ps} for more details.

%
%

%

The remainder of the paper is organized as follows: In \Cref{sec:preliminaries},
we introduce some well-known facts about the  fractional Schr\"odinger equations \eqref{nls}, which will be used throughout the paper. In \Cref{sec:ps}, we present some compactness property of the Palais-Smale sequences. In \Cref{sec:mainresult},
we prove the main result \Cref{thm:main}.

\subsection*{Acknowledgements.}
The authors have been partially supported by the NSF grant of China (No. 11671046, and  No. 11831004).

\section{Preliminaries}
\label{sec:preliminaries}
In this section, we begin with some notation that will be useful throughout this paper.
\subsection{Notation}

There are two ways to define the fractional Laplacian $\slaplace \varphi$ for the real-valued functions
$\varphi\in \dot{H}^{s}\sts{\R^N}$ with $0<s<1$. On the one hand,  the fractional Laplacian of $\varphi$ can be defined by the  Fourier transform as
\begin{equation*}
  \Fcal\sts{\slaplace \varphi}\sts{\xi} = \abs{\xi}^{2s}\sts{\Fcal\varphi}\sts{\xi}.
\end{equation*}
On the other hand, for $\varphi\in \dot{H}^{s}\sts{\R^N}$ with $0<s<1$,
 one can obtain by the fractional heat kernel that
\begin{equation}\label{slaplace}
  \slaplace \varphi\sts{x}
  =
  \frac{1}{C\sts{N,s}}\;
  \int_{\R^N}\frac{ \varphi\sts{x}-\varphi\sts{y} }{ \abs{x-y}^{N+2s} }\d y,
\end{equation}
where
\begin{equation}\label{cstt}
  {C\sts{N,s}} =
   \int_{\R^N}\frac{ 1-\cos\sts{ \eta_1 }}{ \abs{\eta}^{N+2s} }\d\eta.
\end{equation}

Let $O\sts{N}$ be the orthogonal group in $\R^N$, and $G$ be a closed subgroup of the group $O\sts{N}$.
Let $\Z_2=\{1,-1\}$ be the group of 2nd roots of unity, and $\sigma$ be a continuous group homomorphism
from $G$ to $\Z_2$.

For each $x\in\R^N$, let $G\cdot x$ denote the $G-$orbit of the point $x$,
and $G_{x}$ denote the stabilizer subgroup of the group $G$ with respect to the point $x$
, i.e.
\begin{equation*}
  G\cdot x = \setcdt{gx}{g\in G},
  \quad\text{and}\quad
  G_{x}= \setcdt{g\in G}{gx=x}.
\end{equation*}
The domain $\Omega$ in $\R^N$ is said to be \textit{$G$-invariant}, if for each $x\in \Omega$, $G\cdot x\subseteq \Omega$.
For any $G$-invariant domain $\Omega$, we denote
\begin{equation*}
  \Omega^{G}:=
  \setcdt{ x\in\Omega }{  ~~gx=x \text{~~for all~~} g\in G }.
\end{equation*}
Any function $u$ on the $G$-invariant domain $\Omega$ is said to be \textit{$\sigma$-equivariant} if
\begin{equation*}
  u\sts{g x}=\sigma\sts{g}u\sts{x},\quad\text{for all~~} g\in G \text{~~and~~} x\in\Omega.
\end{equation*}

Now for each $s$ with $0<s<1$,
we can obtain the representation for the norm on $\dot{H}^{s}\sts{\R^N}$ by the fractional heat kernel (see \cite{LiebLoss:book, DPV2012, tartar2007} for instance):
\begin{equation}
\label{norm}
    \norm{u}^2
    := \iint_{\R^N\times \R^N}
  \frac{
    \abs{ u\sts{x}-u\sts{y} }^2
  }{
    \abs{x-y}^{N+2s}
  }\dx\dy
  =
  {2 C\sts{N,s}}\int_{\R^N} \abs{\xi}^{2s}\abs{\sts{\Fcal u}\sts{\xi}}^2\d\xi,
\end{equation}
where
$
   {C\sts{N,s}}
$ is defined by \eqref{cstt}.
Moreover, for $0<s<\frac{N}{2}$ and any $u\in \dot{H}^{s}\sts{\R^N}$,  by the Sobolev embedding inequality in \cite{ CT2004jmaa,FJX2018cvpde,FL2010cvpde, LiebLoss:book}, we have
$
\dot{H}^{s}\sts{\R^N}\hookrightarrow L^{ 2^{\ast}_{s} }\sts{\R^N}.
$ More preceisely,
\begin{align}\label{sblv:sharp}
\sts{\int_{\R^N}\abs{u\sts{x}}^{ 2^{\ast}_{s} }\dx }^{ \frac{ 2 }{ 2^{\ast}_{s} } }
\leq
\mathrm{S}\sts{N,s} \iint_{\R^N\times \R^N}\frac{ \abs{ u\sts{x}-u\sts{x} }^2 }{ \abs{x-y}^{N+2s} }\dx\dy,
\end{align}
where
\begin{equation}\label{cstt:sharp}
\mathrm{S}\sts{N,s}
=
2^{-2s}\pi^{-s}\frac{ \Gamma\sts{ \frac{ N-2s }{ 2 } } }{ \Gamma\sts{ \frac{ N+2s }{ 2 } } }\sts{ \frac{ \Gamma\sts{N} }{ \Gamma\sts{\frac{N}{2}} } }^{\frac{2s}{N}}.
\end{equation}
For any domain $\Omega \subset \R^N$ with smooth boundary,
let $\dot{H}^{s}\sts{\Omega}$ denote the Sobolev space
which is defined as the completion of $C_{c}^{\infty}\sts{\Omega}$ under the norm which is defined by \eqref{norm}. Since
$\partial\Omega$ is smooth, we have (see for instance \cite{BS2019ampa,grisvard2011,tartar2007})
\begin{equation*}
  \dot{H}^s\sts{\Omega}=\setcdt{ u\in \dot{H}^s\sts{\R^N} }{ u\sts{x}=0\text{~~for~~a.e.~~} x\in\R^N\setminus\Omega   }.
\end{equation*}

For any closed subgroup $G$ of $O\sts{N}$
and
any continuous group homomorphism $\sigma:G\mapsto\Z_2$,
we define the subspace of $ \dot{H}^s\sts{\Omega} $ which coincides with all
$\sigma$-equivariant functions under the group $G$ as follows,
\begin{equation*}
  \dot{H}^s\sts{\Omega}^{\sigma}_{G}
  =
  \setcdt{
    u\in \dot{H}^s\sts{\Omega}
  }{
    u\sts{g x}=\sigma\sts{g}u\sts{x},\text{~~for all~~} g\in G \text{~~and~~} x\in\Omega
  }.
\end{equation*}
In what follows, we will always assume that the group homomorphism $\sigma$ is surjective,
and the group $G$ satisfies
\begin{enumerate}[label=$\bm{(}$\textbf{G}\textbf{\arabic{*}}$\bm{)}$, ref=$\bm{(}$\textbf{G}\textbf{\arabic{*}}$\bm{)}$]
  \item\label{assum:G:orbit} For every $x\in\R^N$, either $\dim\sts{G\cdot x}>0$ or $G\cdot x={x}$.
  \item\label{assum:G:null} There exists at least one point $\xi\in\R^N$ such that $\sigma\sts{G_{\xi}}=\{1\}$.
\end{enumerate}


\begin{lemm}[{\cite[page 195]{BCM2005tmna}}]
	\label{lem:ifntdim}
	Let $G$ be  closed subgroup of the group $O\sts{N}$, $\Omega$ be a $G$-invariant domain in $\R^N$, and $\sigma:G\mapsto\Z_2$ be a continuous group homomorphism.
	If the group $G$ satisfies \ref{assum:G:null}, then
	the space $\dot{H}^s\sts{\Omega}^{\sigma}_{G}$ is infinite dimensional.
\end{lemm}

\subsection{The $\sigma$-equivariant solutions of \eqref{nls} vanishing outside a $G$-invariant domain}
Let $G$ be  closed subgroup of the group $O\sts{N}$, $\Omega$ be a $G$-invariant domain in $\R^N$, and $\sigma:G\mapsto\Z_2$ be a continuous group homomorphism.
Now, we  consider
the fractional critical Schr\"odinger equation,
\begin{align}\label{dnls}
\begin{cases}
  \slaplace{u}-\abs{u}^{2^{\ast}_{s}-2}u=0, & \mbox{in } \R^N, \\
  u\in\dot{H}^s\sts{\Omega}^{\sigma}_{G}. & \
\end{cases}
\end{align}

By the critical point theory (for instance, see \cite{Chang2005, Struwe2008, Willem1997}),
the function $u$ satisfies \eqref{dnls} if and only if $u\in\dot{H}^s\sts{\Omega}^{\sigma}_{G}$ is a critical point of the Lagrange functional as follows:
\begin{equation}\label{func:e}
  \Ecal\sts{u\, ; \,\Omega }
  =
  \frac{1}{2}\iint_{\R^N\times \R^N}\frac{ \abs{ u\sts{x}-u\sts{y} }^2 }{ \abs{x-y}^{N+2s}}\dx\dy
  -
  \frac{1}{ 2^{\ast}_{s} }\int_{ \Omega }\abs{u\sts{x}}^{ 2^{\ast}_{s} }\dx.
\end{equation}
Moreover, if $u$ is a nontrivial solution to \eqref{dnls}, then $u$ also belongs to the \textit{Nehari manifold} $\Nscr\sts{\Omega}^{\sigma}_{G}$,
\begin{equation}
\label{nehari:omega}
  \Nscr\sts{\Omega}^{\sigma}_{G}
  =
  \setcdt{u
  }{
  u\in \dot{H}^s\sts{\Omega}^{\sigma}_{G}\setminus\ltl{0},  \Ncal\sts{u\, ; \,\Omega }=0
     },
\end{equation}
where the \textit{Nehari functional} $  \Ncal\sts{u\, ; \,\Omega }$ is defined by
\begin{equation}\label{func:n}
  \Ncal\sts{u\, ; \,\Omega }
  =
  \iint_{\R^N\times \R^N}\frac{ \mid u\sts{x}-u\sts{y}\mid^{2} }{ \abs{x-y}^{N+2s}}\dx\dy
  -
  \int_{\Omega}\abs{u\sts{x}}^{ 2^{\ast}_{s} }\dx.
\end{equation}
Now, we can reduce the variational problem for \eqref{dnls}  to seek the critical points of $  \Ecal\sts{\cdot\, ; \,\Omega } $
restricted to the subspace $\dot{H}^s\sts{\Omega}^{\sigma}_{G}$ by the following lemma.
\begin{lemm}
\label{lem:weaksln}
Under the assumptions of \Cref{lem:ifntdim},
if $u\in \dot{H}^s\sts{\Omega}^{\sigma}_{G}$ satisfies
\begin{equation*}
  \action{ \nabla\Ecal\sts{u\, ; \,\Omega } }{ \varphi }=0,
  \quad
  \text{~~for all~~} \varphi\in C_{c}^{\infty}\sts{\Omega}^{\sigma}_{G},
\end{equation*}
where
\begin{equation}
\label{eq:ccifnt}
  C_{c}^{\infty}\sts{\Omega}^{\sigma}_{G}
  =
  \setcdt{
    u\in C_{c}^{\infty}\sts{\Omega}
  }{
    u\sts{g x}=\sigma\sts{g}u\sts{x},\text{~~for all~~} g\in G \text{~~and~~} x\in\Omega
  },
\end{equation}
then
\begin{equation*}
  \action{ \nabla\Ecal\sts{u\, ; \,\Omega } }{ \tphi }=0,
  \quad
  \text{~~for all~~} \tphi\in \dot{H}^s\sts{\Omega}.
\end{equation*}
\end{lemm}
\begin{proof}
Let $\widetilde{\varphi}\in \dot{H}^s\sts{\Omega}$. Define
\begin{equation*}
  \varphi(x)=\frac{1}{\mu(G)}\int_{G}\sigma(g)\widetilde{\varphi}(gx)\,d\mu,
\end{equation*}
where $\mu$ is the Haar measure on $G$. Then $\varphi \in C_{c}^{\infty}\sts{\Omega}^{\sigma}_{G}$, and note that $ \action{ \nabla\Ecal\sts{u\, ; \,\Omega } }{ \varphi }=0$, therefore, by Fubini's theorem and  a change of variable, we obtain
\begin{equation*}
\begin{aligned}
0=&\iint_{\R^N\times \R^N}\frac{ ( u\sts{x}-u\sts{y})(\varphi(x)-\varphi(y)) }{ \abs{x-y}^{N+2s}}\dx\dy
  -
\int_{\Omega}\abs{u\sts{x}}^{ 2^{\ast}_{s}-2 }u(x)\varphi(x)\dx\\
=&\frac{1}{\mu(G)}\iint_{\R^N\times \R^N}\int_{G}\frac{ \sigma(g)( u\sts{x}-u\sts{y})(\widetilde{\varphi}(gx)-\widetilde{\varphi}(gy)) }{ \abs{x-y}^{N+2s}}\d\mu\dx\dy\\
&\quad -\frac{1}{\mu(G)}\int_{\Omega}\int_{G}\abs{u\sts{x}}^{ 2^{\ast}_{s}-2 }u(x)\sigma(g)\widetilde{\varphi}(gx)\d\mu\dx\\
=&\frac{1}{\mu(G)}\iint_{\R^N\times \R^N}\int_{G}\frac{ ( u\sts{gx}-u\sts{gy})(\widetilde{\varphi}(gx)-\widetilde{\varphi}(gy)) }{ \abs{x-y}^{N+2s}}\d\mu\dx\dy\\
&\quad -\frac{1}{\mu(G)}\int_{\Omega}\int_{G}\abs{u\sts{gx}}^{ 2^{\ast}_{s}-2 }u(gx)\widetilde{\varphi}(gx)\d\mu\dx\\
=&\frac{1}{\mu(G)}\int_{G}\iint_{\R^N\times \R^N}\frac{ ( u\sts{gx}-u\sts{gy})(\widetilde{\varphi}(gx)-\widetilde{\varphi}(gy)) }{ \abs{x-y}^{N+2s}}\dx\dy\ \d\mu\\
&\quad-\frac{1}{\mu(G)}\int_{G}\int_{\Omega}\abs{u\sts{gx}}^{ 2^{\ast}_{s}-2 }u(gx)\widetilde{\varphi}(gx)\dx\d\mu\\
=&\frac{1}{\mu(G)}\int_{G}\d\mu\iint_{\R^N\times \R^N}\frac{ ( u\sts{x}-u\sts{y})(\widetilde{\varphi}(x)-\widetilde{\varphi}(y)) }{ \abs{x-y}^{N+2s}}\dx\dy\\
&\quad -\frac{1}{\mu(G)}\int_{G}\d\mu\int_{\Omega}\abs{u\sts{x}}^{ 2^{\ast}_{s}-2 }u(x)\widetilde{\varphi}(x)\dx\\
 =&\iint_{\R^N\times \R^N}\frac{ ( u\sts{x}-u\sts{y})(\widetilde{\varphi}(x)-\widetilde{\varphi}(y)) }{ \abs{x-y}^{N+2s}}\dx\dy
 -\int_{\Omega}\abs{u\sts{x}}^{ 2^{\ast}_{s}-2 }u(x)\widetilde{\varphi}(x)\dx\\
=&\action{ \nabla\Ecal\sts{u\, ; \,\Omega } }{ \widetilde{\varphi} }.
\end{aligned}
\end{equation*}
This completes the proof.
\end{proof}

\begin{lemm}
\label{lem:nehari}
Let the functionals $\Ecal\sts{\cdot\, ; \,\Omega }$, $\Ncal\sts{\cdot\, ; \,\Omega }$ be defined by \eqref{func:e}, \eqref{func:n} respectively, and ${\mathrm{S}\sts{N,s}}$
be given by \eqref{cstt:sharp}.
Under the assumptions of \Cref{lem:ifntdim},
the following statements hold:
\begin{enumerate}[label={(}\texttt{\alph{*}}{)}, ref={(}{\alph{*}}{)}]
\item\label{lem:nehari:bd} for any $u\in \dot{H}^s\sts{\Omega}^{\sigma}_{G}\setminus\ltl{0}$ satisfying
    $
      \norm{u}< {\mathrm{S}\sts{N,s}}^{ -\frac{N}{4s} }
    $
we have
    $\Ncal\sts{u\, ; \,\Omega }>0$;
\item\label{lem:pstv:e} for any $u\in \dot{H}^s\sts{\Omega}^{\sigma}_{G}\setminus\ltl{0}$ satisfying
    $ \norm{u}<\sts{ \frac{ 2^{\ast}_{s} }{ 2 }{\mathrm{S}\sts{N,s}}^{ -\frac{ 2^{\ast}_{s} }{ 2 } }  }^{ \frac{1}{ 2^{\ast}_{s}-2 } } $,
    we have
    $\Ecal\sts{u\, ; \,\Omega }>0$.
\item\label{lem:ngtv:nahari} for any $u\in \dot{H}^s\sts{\Omega}^{\sigma}_{G}\setminus\ltl{0}$ satisfying
    $\Ecal\sts{u\, ; \,\Omega }\leq 0$,
    we have
    $\Ncal\sts{u\, ; \,\Omega }< 0$.
\end{enumerate}

\end{lemm}
\begin{proof}\ref{lem:nehari:bd}.
By \eqref{sblv:sharp}, we have
  \begin{align}\label{eq:002}
    \int_{\Omega}\abs{u\sts{x}}^{ 2^{\ast}_{s} }\dx
    \leq
     {\mathrm{S}\sts{N,s}}^{\frac{ 2^{\ast}_{s} }{ 2 }}\norm{u}^{ 2^{\ast}_{s} }
    <
    \norm{u}^{ 2 }.
  \end{align}
  Inserting \eqref{eq:002} into \eqref{func:n}, we obtain that
  \begin{align*}
    \Ncal\sts{u\, ; \,\Omega }>0.
  \end{align*}

\ref{lem:pstv:e}.
    By \eqref{sblv:sharp}, we get
\begin{align*}
  \notag
    \frac{ 1 }{ 2^{\ast}_{s} }\int_{\Omega}\abs{u\sts{x}}^{ 2^{\ast}_{s} }\dx
    \leq
    \frac{ 1 }{ 2^{\ast}_{s} }
    {\mathrm{S}\sts{N,s}}^{\frac{ 2^{\ast}_{s} }{ 2 }}\norm{u}^{ 2^{\ast}_{s} }
    <
    \frac{ 1 }{ 2 }\norm{u}^{ 2 },
  \end{align*}
  which implies that $\Ecal\sts{u\, ; \,\Omega }>0$.

\ref{lem:ngtv:nahari}.
Since for any $u\in \dot{H}^s\sts{\Omega}^{\sigma}_{G}\setminus\ltl{0}$ with
$
  \Ecal\sts{u\, ; \,\Omega }\leq 0
$,
we have
$
  \frac{2^{\ast}_{s}}{2}\norm{u}^2 \leq \int_{\Omega}\abs{u\sts{x}}^{ 2^{\ast}_{s} }\dx.
$
Therefore, we obtain
\begin{align*}
  \Ncal\sts{u\, ; \,\Omega }
  \leq
  &
  \sts{ 1- \frac{2^{\ast}_{s}}{2} }\norm{u}^2 < 0.
\end{align*}
This ends the proof of \Cref{lem:nehari}.
\end{proof}
As a corollary of \Cref{lem:nehari} that under the assumptions of \Cref{lem:ifntdim}, for any $u\in \dot{H}^s\sts{\Omega}^{\sigma}_{G}\setminus\ltl{0}$ satisfying
$\Ncal\sts{u\, ; \,\Omega }=0$, we have
$
  \Ecal\sts{u\, ; \,\Omega }>0,
$
which enables us to minimize the functional $ \Ecal\sts{\cdot\, ; \,\Omega } $ constrained on the Nehari manifold $ \Nscr\sts{\Omega }^{\sigma} _{G}$. More precisely, let us define
\begin{equation}
\label{m:nehari}
  \mathrm{m}\sts{\Omega}^{\sigma}_{ \Nscr,G } = \inf\setcdt{ \Ecal\sts{u\, ; \,\Omega } }{ u\in \Nscr\sts{\Omega }^{\sigma}_{G} },
\end{equation}
then $\mathrm{m}\sts{\Omega}^{\sigma}_{ \Nscr,G }\geq 0$.
In fact, the Nehari manifold $\Nscr\sts{\Omega }^{\sigma}_{G}$ is a natural constraint for the minimizer to the functional  $\Ecal\sts{\cdot\, ; \,\Omega }$. More precisely, we have :
\begin{lemm}\label{lem:natcstr}
Let $\mathrm{m}\sts{\Omega}^{\sigma}_{ \Nscr,G }$ be defined by \eqref{m:nehari}. Under the assumptions of \Cref{lem:ifntdim},
if $\varphi\in \Nscr\sts{\Omega }^{\sigma}_{G}$ satisfies
$
    \Ecal\sts{\varphi\, ; \,\Omega } = \mathrm{m}\sts{\Omega}^{\sigma}_{ \Nscr,G },
$
then
$$\action{ \grad{\Ecal}\sts{\varphi\, ; \,\Omega } }{ u }=0, \text{~~for~~all~~} u\in \dot{H}^s\sts{\Omega}^{\sigma}_{G}.$$
\end{lemm}
\begin{proof}
  Since $\varphi\in \Nscr\sts{\Omega }^{\sigma}_{G}$ is a minimizer of the functional $\Ecal\sts{u\, ; \,\Omega }$ subject to $ \Ncal\sts{\varphi\, ; \,\Omega }=0 $,
there exists a Lagrange multiplier $\lambda\in\R$ such that
\begin{equation}\label{eq:005}
  \action{ \grad{\Ecal}\sts{\varphi\, ; \,\Omega } }{ u }=\lambda \action{ \grad{\Ncal}\sts{\varphi\, ; \,\Omega } }{ u },
  \quad
  \text{ for all } u\in \dot{H}^s\sts{\Omega}^{\sigma}_{G}.
\end{equation}
By choosing $u=\varphi$ in \eqref{eq:005}, and using the fact that
$$  \action{ \grad{\Ecal}\sts{\varphi\, ; \,\Omega } }{ \varphi } = {\Ncal}\sts{\varphi\, ; \,\Omega } ,$$
we have
\begin{align}\label{eq:006}
\notag
  0 & = \action{ \grad{\Ecal}\sts{\varphi\, ; \,\Omega } }{ \varphi }
  \\
\notag
  & = \lambda \action{ \grad{\Ncal}\sts{\varphi\, ; \,\Omega } }{ \varphi }
  \\
  & = \lambda\sts{ 2\norm{ \varphi }^2 - 2^{\ast}_{s} \int_{\Omega}\abs{\varphi\sts{x}}^{ 2^{\ast}_{s} }\dx  },
\end{align}
which, together with ${\Ncal}\sts{\varphi\, ; \,\Omega }=0$ and $\varphi\neq 0$, implies that
$
  \lambda = 0.
$
This completes the proof.
\end{proof}

The following lemma shows that  $\mathrm{m}\sts{\Omega}^{\sigma}_{ \Nscr,G } $  defined by \eqref{m:nehari}
coincides with the critical value which is characterized via the well-known mountain pass theorem.
\begin{lemm}
\label{lem:hhr2mp}
Suppose that the assumptions of \Cref{lem:ifntdim} hold, and let
\begin{equation}
\label{eq:016}
  \mathrm{m}\sts{\Omega}_{ \mathrm{MP},G }^{\sigma}=\inf_{\gamma\in \Theta}\max_{t\in\left[0,1\right] }\Ecal\sts{\gamma\sts{t}\, ; \,\Omega }.
\end{equation}
where
\begin{equation*}
\Theta =\setcdt{ \gamma\in \Ccal\sts{ \left[0,1\right]\, ,\, \dot{H}^s\sts{\Omega}^{\sigma}_{ G } } }{ \gamma\sts{0}=0,\, \gamma\sts{1}\neq 0,\,  \Ecal\sts{\gamma\sts{1}\, ; \,\Omega }\leq 0 },
\end{equation*}
then we have
\begin{equation*}
  \mathrm{m}\sts{\Omega}_{ \mathrm{MP},G }^{\sigma}=  \mathrm{m}\sts{\Omega}^{\sigma}_{ \Nscr,G }.
\end{equation*}
\end{lemm}
\begin{proof}
Firstly, we show that
\begin{equation*}
  \mathrm{m}\sts{\Omega}^{\sigma}_{ \Nscr,G } \geq \mathrm{m}\sts{\Omega}_{ \mathrm{MP},G }^{\sigma}.
\end{equation*}
For each $u\in \Nscr\sts{\Omega}^{\sigma}_{G}$, we define
$
  \gamma_{\Nscr,u}\sts{t}\doteq\sts{ \frac{ 2^{\ast}_{s} }{ 2 } }^{\frac{ 1 }{ 2^{\ast}_{s}-2 } }t\cdot u,
$
and
$
  \Theta_{\Nscr}
  \doteq
  \setcdt{ \gamma_{\Nscr,u} }{ u\in {\Nscr\sts{\Omega}^{\sigma}_{G}} }.
$
Obviously, we have
\begin{align*}
  \gamma_{\Nscr,u}\in  \Ccal\sts{ \left[0,1\right]\, ,\, \dot{H}^s\sts{\Omega}^{\sigma}_{G} }
  \text{~~with~~}
  \gamma_{\Nscr,u}\sts{0}=0,
  \text{~~and~~}
  \gamma_{\Nscr,u}\sts{1}\neq 0.
\end{align*}
On the one hand, a direct computation shows that
\begin{align*}
  \Ecal\sts{ \gamma_{\Nscr,u}\sts{1} \, ; \,\Omega }
  = &
  \sts{ \frac{ {2^{\ast}_{s}}^{ { 2 } } }{ 2^{ { 2^{\ast}_{s} } } } }^{ \frac{ 1 }{ 2^{\ast}_{s}-2 } }\Ncal\sts{u}=0,
\end{align*}
which implies that,
$
  \Theta_{\Nscr} \subseteq \Theta.
$
On the other hand, for all $t\in\left[0,1\right]$, we have
\begin{align*}
  \Ecal\sts{ \gamma_{\Nscr,u}\sts{t} \, ; \,\Omega }
  =&
  \sts{ \frac{ {2^{\ast}_{s}}^{ { 2 } } }{ 2^{ { 2^{\ast}_{s} } } } }^{ \frac{ 1 }{ 2^{\ast}_{s}-2 } }
  \norm{  u }^2
  \sts{
    t^2
      -
    t^{ 2^{\ast}_{s} }
  }.
\end{align*}
By the elementary fact that
\begin{equation*}
  t^2
      -
    t^{ 2^{\ast}_{s} }
  \leq
    \sts{ { \sts{\frac{2}{2^{\ast}_{s}} }^{ \frac{1}{ 2^{\ast}_{s}-2 } } } }^2
      -
    \sts{ { \sts{\frac{2}{2^{\ast}_{s}} }^{ \frac{1}{ 2^{\ast}_{s}-2 } } } }^{ 2^{\ast}_{s} },
\end{equation*}
we have
\begin{align*}
  \Ecal\sts{ \gamma_{\Nscr,u}\sts{t} \, ; \,\Omega }
  \leq
  &
  \Ecal\sts{ \gamma_{\Nscr,u}\sts{ \sts{\frac{2}{2^{\ast}_{s}} }^{ \frac{1}{ 2^{\ast}_{s}-2 } } } \, ; \,\Omega }
  \\
  =
  &
  \Ecal\sts{ u \, ; \,\Omega }.
\end{align*}
Hence,
\begin{equation}
\begin{aligned}[b]
\label{eq:003}
  \mathrm{m}\sts{\Omega}^{\sigma}_{ \Nscr,G }
  =
  &
  \inf\setcdt{ \Ecal\sts{u\, ; \,\Omega } }{ u\in \Nscr\sts{\Omega }^{\sigma}_{G} }
  \\
  =
  &
  \inf_{ \gamma_{\Nscr,u}\in \Theta_{\Nscr} }\max_{ t\in\left[0,1\right] }
  \Ecal\sts{ \gamma_{\Nscr,u}\sts{t} \, ; \,\Omega }
  \\
  \geq
  &
  \inf_{ \gamma\in \Theta }\max_{ t\in\left[0,1\right] }
  \Ecal\sts{ \gamma\sts{t} \, ; \,\Omega }
  \\
  =
  &
  \mathrm{m}\sts{\Omega}_{ \mathrm{MP},G }^{\sigma}.
\end{aligned}
\end{equation}

Next, we show that
\begin{equation*}
  \mathrm{m}\sts{\Omega}_{ \mathrm{MP},G }^{\sigma}\geq \mathrm{m}\sts{\Omega}^{\sigma}_{ \Nscr,G }.
\end{equation*}
Indeed, for each $ \gamma\in \Theta$, one has $\gamma\sts{0}=0$, by \ref{lem:nehari:bd} of \Cref{lem:nehari}, there exists $s_{u}\in\sts{0,1}$
such that for all $t\in \sts{ 0, s_{u} }$,
\begin{equation}
\label{eq:007}
  \Ncal\sts{ \gamma\sts{t} \, ; \,\Omega }>0.
\end{equation}
However $\Ecal\sts{ \gamma\sts{1} \, ; \,\Omega }\leq 0$, by \ref{lem:ngtv:nahari} of \Cref{lem:nehari}, we have
\begin{equation}\label{eq:008}
  \Ncal\sts{ \gamma\sts{1} \, ; \,\Omega }<0.
\end{equation}
By  \eqref{eq:007} and \eqref{eq:008}, there exists $t_{u,\max}\in \sts{s_{u},1  }\subset\left[0,1\right]$ such that
$
  \Ncal\sts{ \gamma\sts{ t_{u,\max} } \, ; \,\Omega }=0,
$
which means that $\gamma\sts{ t_{u,\max} } \in \Nscr\sts{\Omega }^{\sigma}_{G}$.  Therefore,
\begin{equation}
\begin{aligned}[b]
\label{eq:004}
  \mathrm{m}\sts{\Omega}_{ \mathrm{MP},G }^{\sigma}
  =
  &
  \inf_{ \gamma\in \Theta }\max_{ t\in\left[0,1\right] }
  \Ecal\sts{ \gamma\sts{t} \, ; \,\Omega }
  \\
  \geq
  &
  \inf_{ \gamma\in \Theta }
  \Ecal\sts{ \gamma\sts{ t_{u,\max} } \, ; \,\Omega }
  \\
  \geq
  &
  \inf\setcdt{ \Ecal\sts{u\, ; \,\Omega } }{ u\in \Nscr\sts{\Omega }^{\sigma}_{G} }
  \\
  =
  &
  \mathrm{m}\sts{\Omega}^{\sigma}_{ \Nscr,G }.
\end{aligned}
\end{equation}
Combining \eqref{eq:003} with \eqref{eq:004}, we hence complete the proof.
\end{proof}
In view of the above result, we will denote for short that
\begin{equation}
\label{eq:015}
  \mathrm{m}\sts{\Omega}_{ G }^{\sigma} :=  \mathrm{m}\sts{\Omega}_{ \mathrm{MP},G }^{\sigma} = \mathrm{m}\sts{\Omega}^{\sigma}_{ \Nscr,G }.
\end{equation}

\begin{lemm}
\label{lem:hari}
If $\Omega$ is a $G$-invariant domain in $\mathbb{R^{N}}$ and $\Omega^{G}\neq \emptyset$, then
\begin{equation*}
  \mathrm{m}\sts{\Omega}_{ G }^{\sigma} = \mathrm{m}\sts{\R^N}_{ G }^{\sigma}.
\end{equation*}
\end{lemm}
%
%
\begin{proof}
On one hand, by the embedding that $\Omega\subset\mathbb{R^{N}}$, we have
$$ \mathrm{m}\sts{\Omega}_{ G }^{\sigma} \geq \mathrm{m}\sts{\R^N}_{ G }^{\sigma}.$$
On the other hand, we fix $x_{_{0}}\in \Omega^{G}$ and choose a sequence $\{\varphi_{n}\}$ in $\Nscr\sts{\mathbb{R^{N}} }^{\sigma}_{G}\cap C_{c}^{\infty}(\mathbb{R^{N}})$ such that $\Ecal\sts{ \varphi_{n} \, ; \,\mathbb{R^{N}} }\rightarrow \mathrm{m}\sts{\R^N}_{ G }^{\sigma}$. Since $\varphi_{n}$ has compact support. we may choose $\lambda_{n}>0$ such that $$\text{supp}\widetilde{\varphi_{n}}=\text{supp}\lambda_{n}^{\frac{-(N-2s)}{2}}\varphi_{n}(\frac{x-x_{0}}{\lambda_{n}})\subset \Omega.$$
As $x_{0}$ is a $G$-fixed point, $\widetilde{\varphi_{n}}$ is $\sigma$-equivariant.
Using the fact that $$\parallel \varphi_{n}\parallel^{2}=\parallel \widetilde{\varphi_{n}}\parallel^{2}$$ and
$$\int_{\R^{N}}\mid \varphi_{n}\mid^{2_{s}^{\ast}}\dx=\int_{\Omega}\mid \widetilde{\varphi_{n}}\mid^{2_{s}^{\ast}}\dx
.$$
We have $\widetilde{\varphi_{n}}\in \Nscr\sts{\Omega} ^{\sigma}_{G}$, hence

  $$\mathrm{m}\sts{\Omega}_{ G }^{\sigma}\leq \Ecal\sts{\widetilde{\varphi_{n}}\, ; \,\Omega }  =\Ecal\sts{\varphi_{n}\, ; \,\mathbb{R^{N}} }\rightarrow \mathrm{m}\sts{\R^N}_{ G }^{\sigma},$$
  hence $ \mathrm{m}\sts{\Omega}_{ G }^{\sigma} \leq \mathrm{m}\sts{\R^N}_{ G }^{\sigma}$, which implies the result.
   \end{proof}

\subsection{Some useful estimates}

\begin{lemm}[{{\cite[Proposition 2.9]{BSY2018dcds}}}]
\label{lem:caccippoli}
Let $ f \in \sts{ \dot{H}^{s}\sts{\Omega} }' $, and if $u\in  \dot{H}^{s}\sts{\Omega}$ satisfies that
\begin{equation*}
  \iint_{\R^N\times \R^N}\frac{ \sts{ u\sts{x}-u\sts{y} }\sts{ \varphi\sts{x}-\varphi\sts{y} } }{ \abs{x-y}^{N+2s} }\dx\dy
  =
  \action{ f }{ \varphi },\quad\text{ for any } \varphi\in \dot{H}^{s}\sts{\Omega},
\end{equation*}
then for any open subset $\tilde{\Omega} $ of $\R^N$ with $ \tilde{\Omega}\cap \Omega\neq\emptyset $ and any nonnegative function
$ \phi\in C_{0}^{\infty}\sts{ \tilde{\Omega} } $, we have
\begin{align*}
  &
  \iint_{\tilde{\Omega} \times \tilde{\Omega}} \frac{ \abs{ u\sts{x}\phi\sts{x}-u\sts{y}\phi\sts{y} }^2 }{ \abs{x-y}^{N+2s} }\dx\dy
  \\
   \leq
   &
   C_{0}
   \iint_{\tilde{\Omega} \times \tilde{\Omega} } \frac{ \abs{ \phi\sts{x}-\phi\sts{y} } }{ \abs{x-y}^{N+2s} }\sts{ \abs{u\sts{x}}^2+\abs{u\sts{y}}^2 }\dx\dy
   \\
   &
   +
   C_{0}
   \sts{ \sup_{ y\in \supp{ \phi } }\int_{\R^N\setminus\tilde{\Omega}}\frac{ \abs{u\sts{x}} }{ \abs{ x-y }^{N+2s} }\dx }
   \int_{\tilde{\Omega} }\abs{u\sts{x}}\sts{\phi\sts{x}}^2\dx
   +
   C_{0}\abs{ \action{ f }{ u \varphi^2 } },
\end{align*}
where $C_{0}$ is an absolute constant.
\end{lemm}

\begin{prop}[{\cite[Proposition 2.3]{BP2016adv}}]
\label{prop:lcestmt}
Let $0<r<R$. If $ u\in\dot{H}^s\sts{ B\sts{0,r}  } $, then there exists a positive constant $C\sts{ N,s,\frac{R}{r} }$ such that
\begin{equation*}
  \sts{ \int_{ B\sts{0,r} }\abs{u\sts{x}}^{2^{\ast}_{s} }\dx }^{ \frac{ 2 }{2^{\ast}_{s} } }
  \leq
  C\sts{ N,s,\frac{R}{r} }
  \iint_{B\sts{0,R}\times B\sts{0,R}}\frac{ \abs{ u\sts{x}-u\sts{y} }^2 }{ \abs{x-y}^{N+2s} }\dx\dy,
\end{equation*}
where the constant $C\sts{ N,s,\frac{R}{r} }$ goes to $\infty$ as $R$ goes to $r$.
\end{prop}

In order to describe the behaviour of Palais–Smale sequences for the variational problem in \eqref{eq:016},
we define L\'{e}vy's concentration function for any $u\in L^{2^{\ast}_{s}}\sts{\Omega}$ as follows,
\begin{equation*}
  \Qcal_{u}\sts{r}:= \sup_{ z\in\R^N }\int_{ \ball{z}{r} }\abs{ u\sts{x} }^{ 2^{\ast}_{s} }\dx.
\end{equation*}
We collect here some facts about L\'{e}vy's concentration function.
\begin{prop}
\label{prop:levy}
Let $\Omega$ be a bounded domain of $\R^N$. If $u\in L^{2^{\ast}_{s}}\sts{\Omega}$ with $u\sts{x}=0$ a.e. for $x\in\R^N\setminus\Omega$, then the following statements hold.
\begin{enumerate}[label=${(}${\roman{*}}${)}$, ref=${(}${\roman{*}}${)}$]
\item
\label{levy:1}
 For any $\delta$ with $0<\delta<\int_{\Omega}\abs{u\sts{x}}^{ 2^{\ast}_{s} }\dx$, there exists $r>0$, $z_{\max}\in\R^N$ such that
    \begin{equation*}
      \int_{ \ball{z_{\max}}{r} }\abs{ u\sts{x} }^{ 2^{\ast}_{s} }\dx
      =
      \delta,
    \end{equation*}
    where
$
      \dist{z_{\max}}{\Omega}\leq r.
$
\item
\label{levy:2}
 For any $r>0$ and $\xi\in\R^N$,  we have
    \begin{equation*}
      \Qcal_{u_{r,\xi}}\sts{1} = \Qcal_{u}\sts{r},
    \end{equation*}
    where
    $
          u_{r,\xi}\sts{x} = r^{ \frac{N-2s}{2} }u\sts{ r x+\xi }.
    $
\item
\label{levy:3}
 For any $0<r<R$, we have
    \begin{equation}
    \label{eq:018}
      \Qcal_{u}\sts{R}\leq \lfloor\sts{N+1}\frac{R}{r}\rfloor\Qcal_{u}\sts{r}.
    \end{equation}
\end{enumerate}
\end{prop}

\begin{proof}
\ref{levy:1}. The proof is standard, please  refer  to \cite[Lemma 3.1]{BSY2018dcds} for the details.

\ref{levy:2}.  For any $z\in\R^N$ , we have by the change of variables that
\begin{align}
\label{eq:017}
\int_{ \ball{z}{1} }\abs{ u_{r,\xi}\sts{x} }^{ 2^{\ast}_{s} }\dx
  =&
    \int_{ \ball{rz+\xi}{r} }\abs{ u\sts{x} }^{ 2^{\ast}_{s} }\dx.
\end{align}
Taking the supremum over $\R^N$ on both sides of \eqref{eq:017}, we obtain that
\begin{align*}
  \Qcal_{u_{r,\xi}}\sts{1}
  =&
  \sup_{ z\in\R^N }\int_{ \ball{z}{1} }\abs{ u_{r,\xi}\sts{x} }^{ 2^{\ast}_{s} }\dx
  \\
  =&
  \sup_{ z\in\R^N }\int_{ \ball{rz+\xi}{r} }\abs{ u\sts{x} }^{ 2^{\ast}_{s} }\dx
  \\
  =&
  \sup_{ z\in\R^N }\int_{ \ball{z}{r} }\abs{ u\sts{x} }^{ 2^{\ast}_{s} }\dx
  \\
  =&\;
  \Qcal_{u}\sts{r}.
\end{align*}

\ref{levy:3}. Let $x\in\R^N$ and $R>0$,  there exists $\lfloor\sts{N+1}\frac{R}{r}\rfloor$ balls
${\ball{z_1}{r}}$, $\ball{z_2}{r}$, $\cdots$, ${\ball{z_{\lfloor\sts{N+1}\frac{R}{r}\rfloor}}{r}} $  in $\R^N$ such that their union
cover $\ball{z}{R}$ (see for instance \cite[Corollary 1.3]{G2018dcg} ), therefore
\begin{equation*}
  \int_{ \ball{z}{R} }\abs{u\sts{x}}^{ 2^{\ast}_{s} }\dx
  \leq
  \sum_{j=1}^{ \lfloor\sts{N+1}\frac{R}{r}\rfloor }\int_{ \ball{z_j}{r} }\abs{u\sts{x}}^{ 2^{\ast}_{s} }\dx,
\end{equation*}
which implies that
\begin{equation}
\label{eq:019}
  \int_{ \ball{z}{R} }\abs{u\sts{x}}^{ 2^{\ast}_{s} }\dx
  \leq
  { \lfloor\sts{N+1}\frac{R}{r}\rfloor }\Qcal_{u}\sts{r}.
\end{equation}
Taking the supremum to the left hand side of \eqref{eq:019}, we can obtain  \eqref{eq:018}.

This ends the proof of \Cref{prop:levy}.
\end{proof}

\section{Palais-Smale sequences}
\label{sec:ps}
As a consequence of a general minimax method, using Ekeland's $\varepsilon$-variational principle (see, for instance \cite{AM2007,Willem1997}), 
we have the following well-known result.
\begin{lemm}
\label{hari:cd}
Suppose that  the assumptions of \Cref{lem:ifntdim} hold. Let $\mathrm{m}\sts{\Omega}_{ G }^{\sigma}$ be defined by \eqref{eq:015}. Then there exists a sequence $ \left\{ u_n \right\}_{n=1}^{\infty}\subseteq \dot{H}^s\sts{\Omega}_{G}^{\sigma}$ such that
\begin{align*}
   \Ecal\sts{ u_n  \, ; \,\Omega }\rightarrow \mathrm{m}\sts{\Omega}_{ G }^{\sigma},
   \text{~~and~~}
   \nabla\Ecal\sts{ u_n  \, ; \,\Omega }\rightarrow 0 \text{  in  } \sts{ \dot{H}^s\sts{\Omega}_{G}^{\sigma} }{'},
   \text{~~as~~} n\rightarrow \infty.
\end{align*}
\end{lemm}
\begin{proof}
  By Ekeland's $\varepsilon$-variational principle, there exists a sequence
$ \left\{ u_n \right\}_{n=1}^{\infty}\subseteq \dot{H}^s\sts{\Omega}_{G}^{\sigma}$ satisfying 
\begin{align*}
   \Ecal\sts{ u_n  \, ; \,\Omega }\rightarrow \mathrm{m}\sts{\Omega}_{ \mathrm{MP},G }^{\sigma},
   \text{~~and~~}
   \nabla\Ecal\sts{ u_n  \, ; \,\Omega }\rightarrow 0 \text{  in  } \sts{ \dot{H}^s\sts{\Omega}_{G}^{\sigma} }{'},
   \text{~~as~~} n\rightarrow \infty.
\end{align*}
By \eqref{eq:015}, we obtain that  
\begin{align*}
   \Ecal\sts{ u_n  \, ; \,\Omega }\rightarrow \mathrm{m}\sts{\Omega}_{ G }^{\sigma},
   \text{~~and~~}
   \nabla\Ecal\sts{ u_n  \, ; \,\Omega }\rightarrow 0 \text{  in  } \sts{ \dot{H}^s\sts{\Omega}_{G}^{\sigma} }{'},
   \text{~~as~~} n\rightarrow \infty.
\end{align*}
This ends the proof of \Cref{hari:cd}.
\end{proof}
By the change of variables, we have the following result.
\begin{lemm}
\label{lem:sclivrc}
Under the assumptions of \Cref{lem:ifntdim}, for any positive number $\lambda$ and any point $\xi$ in $\R^N$ satisfying
$ G\cdot \xi=\{\xi\}$, let us define the domain
\begin{equation*}
  \Omega_{\xi,\lambda}\doteq\setcdt{ x }{ x\in \frac{1}{\lambda}\sts{ \Omega-\xi } },
\end{equation*}
and the function
\begin{equation*}
  u_{\xi,\lambda}\sts{x}
  \doteq
  \lambda^{ \frac{N-2s}{2} }u\sts{ \lambda x+\xi }.
\end{equation*}
Then the following statements hold :
\begin{enumerate}[label=${(}${\Roman{*}}${)}$, ref=${(}${\Roman{*}}${)}$]
\item $u\in \dot{H}^s\sts{\Omega}_{ G }^{\sigma}$ if and only if
    $u_{\xi,\lambda}\in \dot{H}^s\sts{ \Omega_{\xi,\lambda} }_{ G }^{\sigma}$, moreover, we have
    \begin{equation*}
      \norm{ u }=\norm{ u_{\xi,\lambda} },
      \quad
      \text{and}
      \quad
      \int_{ \Omega }\abs{ u\sts{x} }^{ 2^{\ast}_{s} }\dx
      =
      \int_{ \Omega_{\xi,\lambda} }\abs{ u_{\xi,\lambda}\sts{x} }^{ 2^{\ast}_{s} }\dx.
    \end{equation*}
\item if $u\in \dot{H}^s\sts{\Omega}_{ G }^{\sigma}$, then we have
\begin{equation*}
  {{\sup}}_{\substack{ \phi\in \dot{H}^s\sts{\Omega}_{ G }^{\sigma}\\ \norm{ \phi }=1 }  }\action{ \grad{\Ecal}\sts{u\, ; \,\Omega } }{ \phi }
  =
  {{\sup}}_{\substack{ \psi\in \dot{H}^s\sts{\Omega}_{ G }^{\sigma}\\ \norm{ \psi }=1 }  }\action{ \grad{\Ecal}\sts{u_{\xi,\lambda}\, ; \,\Omega_{\xi,\lambda} } }{ \psi }.
\end{equation*}
\end{enumerate}
\end{lemm}
The following geometrical lemma will be used to show \Cref{lem:ps}.
\begin{lemm}[\cite{Clapp2016jde,CS2020cpde}]
\label{lem:geom}
Let $G$ be a closed subgroup of the group $O\sts{N}$, $\{x_n\}_{n=1}^{\infty}$ be a sequence of $\R^N$ and $\{\lambda_n\}_{n=1}^{\infty}$ be a sequence of $\sts{0,\infty}$. If
$G$ satisfies \ref{assum:G:orbit}, then,
\begin{enumerate}
\item[] either there exist a subsequence (still denoted by $\{x_n\}_{n=1}^{\infty}$) and a sequence $\{\xi_n\}_{n=1}^{\infty}$ of $\sts{\R^N}^{G}$, such that
    \begin{equation*}
      \frac{1}{\lambda_n}\dist{G {\color{red} \cdot }  x_n}{\xi_n}\leq N_0,
    \end{equation*}
    where $N_0$ is some positive integer independent of $n$;
\item[] or for any integer $q$, there exist $\delta>0$ and elements $g_1$, $g_2$, $ \cdots $, $g_q$ of $G$, such that for $1\leq i\neq k\leq q$,
    \begin{equation*}
      \frac{1}{\lambda_n}\abs{ g_{j}x_n-g_{k}x_n }\rightarrow \infty\text{~~as~~} n\rightarrow \infty.
    \end{equation*}
\end{enumerate}
\end{lemm}
Now, we are able to describe the compactness properties of Palais-Smale sequences for \eqref{eq:015}.
\begin{lemm}
\label{lem:ps}
    Let $G$ be  closed subgroup of the group $O\sts{N}$, $\Omega$ be a $G$-invariant  bounded smooth domain in $\mathbb{R^N}$
    and $\sigma:G\mapsto\Z_2$ be a continuous group homomorphism.
    Suppose the group $G$ satisfies \ref{assum:G:orbit}-\ref{assum:G:null}.
If $ \left\{u_n\right\}_{n=1}^{\infty} $ is a sequence of $\dot{H}^s\sts{{\Omega}}_{ G }^{\sigma}$ satisfying
\begin{align}
\label{ps-sqs}
   \Ecal\sts{ u_n  \, ; \,{\Omega} }\rightarrow \mathrm{m}\sts{{\Omega}}_{ G }^{\sigma},
   \text{ and }
   \nabla\Ecal\sts{ u_n  \, ; \,{\Omega} }\rightarrow 0 \text{  in  } \sts{ \dot{H}^s\sts{{\Omega}}_{G}^{\sigma} }{'}
   \text{ as } n\rightarrow \infty,
\end{align}
then, up to a subsequence (still denoted by $\{u_n\}_{n=1}^{\infty}$),
\begin{enumerate}[label=${(}${PS\arabic{*}}${)}$, ref=${(}${PS\arabic{*}}${)}$,wide, labelwidth=!, labelindent=0pt]
  \item\label{ps:compact} either there exists a nontrivial $U\in \Nscr\sts{{\Omega} }_{ G }^{\sigma}$ such that
  \begin{align*}
    \lim_{n\to\infty}\norm{ u_n-U }=0.
  \end{align*}
  \item\label{ps:noncompact} or there exist a sequence $\left\{\xi_n\right\}_{n=1}^{\infty}\subseteq \sts{\R^N}^G$,
    a sequence $\left\{\lambda_n\right\}_{n=1}^{\infty}\subseteq \sts{0,\infty}$ and
    a nontrivial function $V\in \Nscr\sts{\Hcal }_{ G }^{\sigma}$ satisfying
    \begin{equation*}
      \slaplace{V}-\abs{V}^{2^{\ast}_{s}-2}V=0, \quad\text{in } \Hcal,
    \end{equation*}
     such that
    \begin{align*}
      \lim_{n\to\infty}\norm{ u_n\sts{\cdot}-\frac{ 1 }{ \lambda_n^{\frac{ N-2s }{ 2 }} }V\sts{ \frac{ \cdot-\xi_{n} }{ \lambda_n } } }=0,
    \end{align*}
  where $\Hcal$ is either a half space of $\R^N$ or $\Hcal=\R^N$.
\end{enumerate}
\end{lemm}

\begin{proof}
We divide the proof into several steps.
\begin{enumerate}[label=$\textit{\textbf{Step}~~}\textbf{\arabic{*}.}$, ref=$\textit{\textbf{Step}~~}\textbf{\arabic{*}.}$,listparindent=0pt,wide =0.5\parindent,]
  \item[\bf Step 1.] We claim that the sequence $\{u_n\}^{\infty}_{n=1}$ satisfying \eqref{ps-sqs} is uniformly bounded in $ \dot{H}^s\sts{{\Omega}}_{G}^{\sigma} $.
By \eqref{ps-sqs},  we have
\begin{align*}
  \so{1}\norm{u_n}
  =
  \action{ \nabla\Ecal\sts{u_n\, ; \,\Omega } }{ u_n }
  =
  \norm{u_n}^2 - \int_{\Omega}\abs{ u_n\sts{x} }^{ 2^{\ast}_{s} }\dx,
\end{align*}
we obtain that
\begin{align*}
  \frac{s}{N} \norm{u_n}^2
  =&
  { \Ecal\sts{u_n\, ; \,\Omega } }
  -
  \frac{1}{ 2^{\ast}_{s} }
  \action{ \nabla\Ecal\sts{u_n\, ; \,\Omega } }{ u_n }
  \\
  =&
  \mathrm{m}\sts{\Omega}_{G}^{\sigma} + \so{1}\norm{u_n},
\end{align*}
and
\begin{align}
\notag
\label{eq:027}
  \frac{s}{N} \int_{\Omega}\abs{ u_n\sts{x} }^{ 2^{\ast}_{s} }\dx
  =&
  { \Ecal\sts{u_n\, ; \,\Omega } }
  -
  \frac{1}{ 2 }
  \action{ \nabla\Ecal\sts{u_n\, ; \,\Omega } }{ u_n }
  \\
  =&
  \mathrm{m}\sts{\Omega}_{G}^{\sigma} + \so{1}\norm{u_n},
\end{align}
which implies that the sequence $ \{u_n\}_{n=1}^{\infty} $ is uniformly bounded in $ \dot{H}^s\sts{{\Omega}}_{G}^{\sigma} $.
Therefore, by the Rellich–Kondrachov theorem, after passing to a subsequence if necessary, we obtain that
\begin{align}
\label{eq:026}
  u_n\rightharpoonup U &\quad \text{weakly in }  \dot{H}^s\sts{{\Omega}}_{G}^{\sigma};
  \\
\label{eq:020}
  u_n\rightarrow U &\quad \text{ in }  L^p\sts{\Omega} \text{  for all  } 1\leq p< 2^{\ast}_{s};
  \\
  \notag
  u_n\rightarrow U &\quad \text{~almost everywhere in~~}  {\Omega}.
\end{align}
Hence, for each $ \varphi\in \dot{H}^s\sts{{\Omega}}_{G}^{\sigma} $, one has
\begin{equation}
\label{eq:0211}
  \iint_{\R^N\times \R^N}\frac{ \sts{ u_n\sts{x}-u_n\sts{y} }\sts{ \varphi\sts{x}-\varphi\sts{y} } }{ \abs{x-y}^{N+2s} }\dx\dy
  \rightarrow
  \iint_{\R^N\times \R^N} \frac{ \sts{ U\sts{x}-U\sts{y} }\sts{ \varphi\sts{x}-\varphi\sts{y} } }{ \abs{x-y}^{N+2s} }\dx\dy.
\end{equation}
Moreover, by \cite[Theorem A.2]{Willem1997} and  \eqref{eq:020} with $p=2^{\ast}_{s}-1$,  we obtain  that
\begin{equation*}
    \abs{u_n}^{ 2^{\ast}_{s}-2 }u_n  \rightarrow  \abs{U}^{ 2^{\ast}_{s}-2 }U \quad \text{ in }  L^1\sts{\Omega}.
\end{equation*}
Therefore, we have for any $\phi\in C_{c}^{\infty}\sts{\Omega}$ that
\begin{equation*}
  \int_{\Omega}\abs{u_n\sts{x}}^{ 2^{\ast}_{s}-2 }u_n\sts{x}\phi\sts{x}
  \rightarrow
  \int_{\Omega}\abs{U\sts{x}}^{ 2^{\ast}_{s}-2 }U\sts{x}\phi\sts{x}.
\end{equation*}
By density of $C_{c}^{\infty}\sts{\Omega}$ in $\dot{H}^s\sts{\Omega}_{G}^{\sigma}$, we obtain
for any $\varphi\in \dot{H}^s\sts{\Omega}_{G}^{\sigma}$ that
\begin{equation}
\label{eq:022}
  \int_{\Omega}\abs{u_n\sts{x}}^{ 2^{\ast}_{s}-2 }u_n\sts{x}\varphi\sts{x}  \rightarrow  \int_{\Omega}\abs{U\sts{x}}^{ 2^{\ast}_{s}-2 }U\sts{x}\varphi\sts{x},
\end{equation}
which  together with \eqref{ps-sqs} and \eqref{eq:0211} imlies for any $\varphi\in \dot{H}^s\sts{\Omega}_{G}^{\sigma}$ that
\begin{align}
\label{eq:023}
  \action{ \nabla\Ecal\sts{U\, ; \,\Omega } }{ \varphi  } &=0.
\end{align}

\item[\bf Step 2.] For the case $U\neq 0$, we claim that \ref{ps:compact} holds. In fact, on the one hand, by choosing $\varphi=U$ in \eqref{eq:023}, we obtain that
\begin{align*}
  \Ncal\sts{U\, ; \,\Omega }
  =0,
\end{align*}
which means $U\in \Nscr\sts{\Omega }_{G}^{\sigma}$. By \eqref{m:nehari} and \eqref{eq:015}, we obtain that
\begin{equation}
\label{eq:024}
  \frac{ s }{ N }\norm{U}^2
  =
  { \Ecal\sts{U\, ; \,\Omega } }-\frac{1}{ 2^{\ast}_{s} }\Ncal\sts{U\, ; \,\Omega }
  \geq
  \mathrm{m}\sts{\Omega}_{G}^{\sigma}.
\end{equation}
 On the other hand, by the weak lower semicontinuity of the norm $\|\cdot\|$,
we have
\begin{equation}
\label{eq:025}
  \norm{U}^2\leq \lim_{n\to\infty}\norm{u_n}^2 = \frac{N}{s}\mathrm{m}\sts{\Omega}_{G}^{\sigma}.
\end{equation}
Therefore,  we get
\begin{equation*}
  \lim_{n\to\infty}\norm{u_n} = \norm{U},
\end{equation*}
which together with \eqref{eq:026} implies that
\begin{equation*}
  \lim_{n\to\infty}\norm{{ u_n-U }} =0.
\end{equation*}
Hence \eqref{ps:compact} holds.

\item[\bf Step 3.] For the case $U=0$.
By \eqref{eq:027}, up to a subsequence (we still denote by $\{u_n\}_{n=1}^{\infty}$), we get
\begin{equation}\label{eq:009}
  \int_{\Omega}\abs{ u_n\sts{x} }^{ 2^{\ast}_{s} }\dx
  \geq
  \frac{ N }{ 2s }\mathrm{m}\sts{\Omega}_{G}^{\sigma}.
\end{equation}

Let us choose $\delta$ satisfying $0<\delta <\frac{ N }{ 2s }\mathrm{m}\sts{\Omega}_{G}^{\sigma}$ to be determined later.
By \Cref{prop:levy}, there exist a sequence $\{z_n\}_{n=1}^{\infty}$ of $\R^N$ and
 a sequence $\{\lambda_n\}_{n=1}^{\infty}$ of $\sts{0,+\infty}$, such that
 \begin{align*}
  \delta &=\int_{ \ball{z_n}{\lambda_n} } \abs{ u_n\sts{x} }^{ 2^{\ast}_{s} }\dx.
\end{align*}
\item[\bf Step 3a.] We claim that
 after passing to another subsequence if necessary,
 there exists a sequence $\{\xi_n\}_{n=1}^{\infty}$ of $\sts{\R^N}^{G}$
such that
\begin{equation}
\label{eq:028}
  \frac{1}{\lambda_n}\dist{G\cdot z_n}{\xi_n}\leq N_0,
\end{equation}
where the constant $N_0$ is the positive integer in \Cref{lem:geom}.
We suppose by contradiction that \eqref{eq:028} does not hold. On the one hand, by \Cref{lem:geom}, for each $q\in\N$, one can find
$q$ elements $g_1$, $g_2$, $\cdots$, $g_q$ of $G$, such that, for $n$ sufficiently large,
\begin{equation*}
  \abs{ g_j z_n- g_k z_n }\geq 2{ \lambda_n },\text{~~if~~} j\neq k,
\end{equation*}
which implies that
\begin{equation}\label{eq:029}
  \ball{ g_j z_n }{ \lambda_n }\cap \ball{ g_k z_n }{ \lambda_n }=\emptyset.
\end{equation}
On the other hand, using the fact that $\abs{u\sts{gx}}=\abs{u\sts{x}}$ for all $u\in\dot{H}^s\sts{{\Omega}}_{G}^{\sigma}$,
by the change of variables, we have,
\begin{equation}\label{eq:030}
  \int_{\ball{ g_j z_n }{ \lambda_n }}\abs{ u_n\sts{x} }^{ 2^{\ast}_{s} }\dx
  =
  \int_{ \ball{z_n}{\lambda_n} } \abs{ u_n\sts{x} }^{ 2^{\ast}_{s} }\dx,
  \text{~~for all~~} j=1,2,\cdots,q.
\end{equation}
Now, combining \eqref{eq:029} with \eqref{eq:030}, we obtain for each $q\in\N $ that,
\begin{equation*}
  q\delta
  =
  \sum_{j=1}^{q}\int_{\ball{ g_j z_n }{ \lambda_n }}\abs{ u_n\sts{x} }^{ 2^{\ast}_{s} }\dx
  \leq \int_{ {\Omega} }\abs{ u_n\sts{x} }^{ 2^{\ast}_{s} }\dx,
\end{equation*}
which contradicts with the uniform boundedness of $\int_{ {\Omega} }\abs{ u_n\sts{x} }^{ 2^{\ast}_{s} }\dx$, see \eqref{eq:027}.
Therefore, \eqref{eq:028} holds. For each $z_n$, there exist $g_n\in G$ and $\xi_n\in\sts{\R^N}^G$ such that
$\abs{ g_n z_n-\xi_n }<N_0 \lambda_n$, hence we have
\begin{equation}
\label{eq:032}
  \delta
   \leq\int_{ \ball{\xi_n}{\sts{N_0+1}\lambda_n} } \abs{ u_n\sts{x} }^{ 2^{\ast}_{s} }\dx,
\end{equation}
which implies that
\begin{equation}\label{eq:036}
  \abs{ \Omega \cap { \ball{\xi_n}{\sts{N_0+1}\lambda_n} } }>0.
\end{equation}

However, by \Cref{prop:levy}, we obtain that
\begin{equation}\label{eq:031}
\begin{aligned}[b]
  \int_{ \ball{\xi_n}{\sts{N_0+1}\lambda_n} } \abs{ u_n\sts{x} }^{ 2^{\ast}_{s} }\dx
  \leq &
  \Qcal_{u_n}\sts{ \sts{N_0+1}\lambda_n }
  \\
  \leq &
   \sts{N+1}\sts{ N_0+1 }
  \Qcal_{u_n}\sts{ \lambda_n }
  \\
  \leq &
  \sts{N+1}\sts{ N_0+1 }\delta.
\end{aligned}
\end{equation}

From now on, we consider the new sequence $\sts{v_n}_{n=1}^{\infty}$ which is defined by
\begin{equation*}
   v_{n}\sts{x}:= \lambda_n^{ \frac{ N-2s }{ 2 } }u_n\sts{ \lambda_n x+\xi_n }.
\end{equation*}
By letting ${\Omega_n}=\setcdt{ x }{ x\in \frac{1}{\lambda_n}\sts{ \Omega-\xi_n } }$ , we have $v_{n}\in \dot{H}^s\sts{\Omega_n}_{G}^{\sigma}$, and
\begin{align}
\label{eq:034}
  \norm{v_n}&=\norm{u_n},
  \\
  \notag
  \int_{ {\Omega_n} } \abs{ v_n\sts{x} }^{ 2^{\ast}_{s} }\dx
  &=\int_{\Omega} \abs{ u_n\sts{x} }^{ 2^{\ast}_{s} }\dx.
\end{align}
Moreover, by \eqref{eq:032} and \eqref{eq:031}, we have
\begin{equation}
\label{eq:033}
  \delta
  \leq
  \int_{ \ball{0}{N_0+1} } \abs{ v_n\sts{x} }^{ 2^{\ast}_{s} }\dx
  \leq
  \sts{N+1}\sts{ N_0+1 }\delta.
\end{equation}
By the Rellich-Kondrachov theorem, after passing to a subsequence if necessary, we have
\begin{align*}
  v_n\rightharpoonup V &\quad \text{weakly in~~}  \dot{H}^s\sts{\R^N}_{G}^{\sigma};
  \\
  v_n\rightarrow V &\quad \text{in~~}  L^p_{\text{loc}}\sts{\R^N} \text{  for all  } 1\leq p< 2^{\ast}_{s};
  \\
  v_n\rightarrow V &\quad \text{almost everywhere in~~}  {\R^N}.
\end{align*}
\item[\bf Step 3b.] We claim that $V\neq 0$. Arguing by contradiction, we assume that $V= 0$.
Let $h\in C_{c}^{\infty}\sts{\R^N}$ be a radially symmetric function satisfying
\begin{align*}
h\sts{x}=
  \begin{cases}
    1 , & \mbox{if } x\in \ball{0}{ N_0+1 }, \\
    0 , & \mbox{if } x\notin \ball{0}{ 2\sts{N_0+1} }.
  \end{cases}
\end{align*}
By the fact that $v_{n}\in \dot{H}^s\sts{\Omega_n}_{G}^{\sigma}$, and \eqref{eq:033}, we have
\begin{equation*}
  \Omega_n\cap \ball{0}{ N_0+1 }\neq\emptyset.
\end{equation*}
Therefore, by replacing
$ \Omega     $, $ \widetilde{\Omega}        $, $ u   $
and
$ f $
with
$ \Omega_{n} $, $  \ball{0}{ 4\sts{N_0+1} } $, $ v_n $
and
$\abs{ v_n }^{ 2^{\ast}_{s}-2 }v_n + { \nabla\Ecal\sts{v_n\, ; \,\Omega_n } }  $ respectively in \Cref{lem:caccippoli},
we obtain that
\begin{align}\label{eq:010}
  &
  \int_{ \ball{0}{4\sts{N_0+1}} }\int_{ \ball{0}{4\sts{N_0+1}} }
  \frac{
    \abs{ v_n\sts{x}h\sts{x}-v_n\sts{y}h\sts{y} }^2
  }{
    \abs{x-y}^{N+2s}
    }
  \dx\dy
  \\
  \label{eq:011}
  \leq &\;
  C_{0}
  \int_{ \ball{0}{4\sts{N_0+1}} }\int_{ \ball{0}{4\sts{N_0+1}} }
  \frac{
    \abs{ h\sts{x}-h\sts{y} }^2
  }{
    \abs{x-y}^{N+2s}
    }
  \sts{ \abs{ v_n\sts{x} }^2 + \abs{ v_n\sts{y} }^2 }
  \dx\dy
  \\
  \label{eq:012}
  &+
  C_{0}
  \sts{
    \sup_{ y\in \ball{0}{2\sts{N_0+1}} }
    \int_{ \R^N\setminus \ball{0}{4\sts{N_0+1}} }
    \frac{ \abs{ v_n\sts{x} } }{ \abs{x-y}^{N+2s} }
    \dx
  }
  \int_{ \R^N } \abs{v_n\sts{x}}h^2\sts{x}\dx
  \\
  \label{eq:013}
  &+
  C_{0}
  \int_{ \R^N } h^2\sts{x}\abs{ v_n\sts{x} }^{ 2^{\ast}_{s} }\dx
  \\
  \label{eq:014}
  &+
  C_{0}
  \abs{ \action{ \nabla\Ecal\sts{v_n\, ; \,\Omega_{n} } }{ h^2 v_n } }.
\end{align}

Since for any $x\in \ball{0}{4\sts{N_0+1}}$,
\begin{equation}
\label{eq:035}
\begin{aligned}[b]
  &
  \int_{ \ball{0}{4\sts{N_0+1}} } \frac{1}{ \abs{ x-y }^{N+2s-2} }\dy
  \\
  = &
  \int_{ \ball{-x}{4\sts{N_0+1}} } \frac{1}{ \abs{ y }^{N+2s-2} }\dy
  \\
  \leq &
  \int_{ \ball{0}{8\sts{N_0+1}} } \frac{1}{ \abs{ y }^{N+2s-2} }\dy
  \\
  \leq &
  C_{1}\sts{ N,N_0 },
\end{aligned}
\end{equation}
where $ C_{1}\sts{ N,N_0 } $ is a constant depending only on $N$ and $N_0$. Using \eqref{eq:035}, we get
\begin{align*}
  \eqref{eq:011}
  \leq &
  2C_{0}
  \norm{\nabla h}^{2}_{L^{\infty}}
  \int_{ \ball{0}{4\sts{N_0+1}} }
  \abs{ v_n\sts{x} }^2
  \sts{
  \int_{ \ball{0}{4\sts{N_0+1}} }
  \frac{
    1
  }{
    \abs{x-y}^{N+2s-2}
    }
  \dy
  }
  \dx
  \\
  \leq &
  2C_{0}
  \norm{\nabla h}^{2}_{L^{\infty}}C_{1}\sts{ N,N_0 }
  \norm{ v_n }_{ L^2\sts{ \ball{0}{4\sts{N_0+1}} } }
  \\
  = &
  \so{1},
\end{align*}
where we used that
$\lim\limits_{n\to\infty}\int_{ \ball{0}{4\sts{N_0+1}} }\abs{ v_n\sts{x}-V\sts{x} }^2\dx=0$ and $V=0$.

Since for any $x\in R^N\setminus \ball{0}{4\sts{N_0+1}}$ and any $ y\in  \ball{0}{2\sts{N_0+1}}$, we have
$\frac{ \abs{x} }{2}\leq \abs{x-y}\leq \frac{3}{2}\abs{x}$, there exists a positive constant
$ C_{2}\sts{ N,N_0 } $ depending only on $N$ and $N_0$ such that
\begin{align*}
  \int_{R^N\setminus \ball{0}{4\sts{N_0+1}}} \frac{1}{ \abs{x-y}^{ \sts{N+2s}\frac{ 2^{\ast}_{s} }{ 2^{\ast}_{s}-1 } } }\dx
  &\leq
  C_{2}\sts{ N,N_0 },
\end{align*}
which together with the H\"older inequality yields that
\begin{align*}
  \eqref{eq:012}
  = &
  C_{0}
  \sts{
    \sup_{ y\in \ball{0}{2\sts{N_0+1}} }
    \int_{ \R^N\setminus \ball{0}{4\sts{N_0+1}} }
    \frac{ \abs{ v_n\sts{x} } }{ \abs{x-y}^{N+2s} }
    \dx
  }
  \int_{ \R^N } \abs{v_n\sts{x}}h^2\sts{x}\dx
  \\
  \leq &
  C_{0}C_{2}\sts{ N,N_0 }^{\frac{2^{\ast}_{s}-1}{2^{\ast}_{s}}}
  (\int_{{\Omega_n}}\abs{ v_n\sts{x} }^{ 2^{\ast}_{s} }\dx)^{\frac{1}{2^{\ast}_{s}}}
  \int_{ \ball{0}{2\sts{N_0+1}} }\abs{ v_n\sts{x} }\dx
  \\
  = &
  \so{1},
\end{align*}
where we used the uniform boundedness of $ \int_{{\Omega_n}}\abs{ v_n\sts{x} }^{ 2^{\ast}_{s} }\dx $,
$\lim\limits_{n\to\infty}\int_{ \ball{0}{2\sts{N_0+1}} }\abs{ v_n\sts{x}-V\sts{x} }\dx=0$ and $V=0$.

Next, combining \cite[Lemma A.1]{BSY2018dcds} with \eqref{sblv:sharp} and \Cref{lem:sclivrc}, we obtain that
\begin{align*}
  \eqref{eq:014}
  = &
  \so{1}\norm{ h^2 v_n }
  =
  \so{1}\norm{ v_n }.
\end{align*}
Now, we deal with \eqref{eq:013}. By the H\"older inequality and \Cref{prop:lcestmt}, we have
\begin{align*}
  \int_{ \R^N } h^2\sts{x}\abs{ v_n\sts{x} }^{ 2^{\ast}_{s} }\dx
  \leq &
  \sts{ \int_{ \ball{0}{2\sts{N_0+1}} }\abs{ v_n\sts{x} }^{ 2^{\ast}_{s} }\dx }^{ \frac{ 2^{\ast}_{s}-2 }{ 2^{\ast}_{s} } }
  \sts{ \int_{\R^N}\abs{ h\sts{x}v_n\sts{x} }^{ 2^{\ast}_{s} }\dx  }^{ \frac{ 2 }{ 2^{\ast}_{s} } }
  \\
  \leq &
  C\sts{ N,s,2 }
  \sts{ \int_{ \ball{0}{2\sts{N_0+1}} }\abs{ v_n\sts{x} }^{ 2^{\ast}_{s} }\dx }^{ \frac{ 2^{\ast}_{s}-2 }{ 2^{\ast}_{s} } }
  \norm{ h v_n }_{\dot{H}^{s}\sts{ \ball{0}{4\sts{N_0+1}} }}^2
  \\
  \leq &
  C\sts{ N,s,2 }\sts{2\sts{ N_0+1 }\sts{N+1}\delta}^{ \frac{ 2^{\ast}_{s}-2 }{ 2^{\ast}_{s} } }
  \norm{ h v_n }_{\dot{H}^{s}\sts{ \ball{0}{4\sts{N_0+1}} }}^2.
\end{align*}
Next, by choosing
\begin{equation*}
  \delta
  =
  \min
  \left\{
    { \frac{ N }{ 2s }\mathrm{m}\sts{\Omega}^{\sigma}_{G} }
    ,
    \frac{ 1 }{ 2\sts{ N_0+1 }\sts{ N+1 } }\sts{ \frac{ 1 }{ 2 C\sts{ N,s,2 } } }^{ \frac{ 2^{\ast}_{s} }{ 2^{\ast}_{s}-2 } }
  \right\},
\end{equation*}
we have
\begin{equation*}
  \norm{ h v_n }_{\dot{H}^{s}\sts{ \ball{0}{4\sts{N_0+1}} }}^2
  =\so{1}.
\end{equation*}
Therefore,
\begin{align*}
  \int_{ \ball{0}{\sts{N_0+1}} } \abs{ v_n\sts{x} }^{ 2^{\ast}_{s} }\dx
  \leq &
  \int_{ \ball{0}{2\sts{N_0+1}} } \abs{ h\sts{x}v_n\sts{x} }^{ 2^{\ast}_{s} }  \dx
  \\
  \leq &
  \norm{ h v_n }_{\dot{H}^{s}\sts{ \ball{0}{4\sts{N_0+1}} }}^{\frac{ 2^{\ast}_{s} }{ 2 }}.
\end{align*}
which contradicts with \eqref{eq:033}. Therefore $V\neq 0$.

\item[\bf Step 3c.] Without loss of generality, one may assume, after passing to another subsequence if necessary, as $n\rightarrow \infty$

\begin{equation*}
  \xi_n\rightarrow \xi^0,
  \quad
  \lambda_n \rightarrow \lambda^0,
\end{equation*}
with
$ \xi^0\in\sts{\R^N}^G $ and $ \lambda^0\geq 0 $.
%
%
Now, we distinguish two cases according to the fact whether
$ \frac{ 1 }{ \lambda_n }\lim\limits_{n\to\infty}\dist{ \xi_n }{\partial\Omega}$ is finite
or not.
   \item[\bf Case I.] $ \frac{ 1 }{ \lambda_n }\lim\limits_{n\to\infty}\dist{ \xi_n }{\partial\Omega}=\infty $.
  In this case, by \eqref{eq:036}, we have $\xi_n\in\Omega$.
  therefore, for each compact subset $F$ of $\R^N$, there exists $n_0$ such that
  \begin{equation*}
    F\subseteq \Omega_n\quad\text{ for all } n\geq n_0.
  \end{equation*}
  On the one hand, for each $ \varphi\in  C_{c}^{\infty}\sts{\R^N}^{\sigma}_{G}$, for $n$ sufficiently large, we obtain that
  \begin{align*}
    &
    \action{ \nabla\Ecal\sts{v_n\, ; \,\R^N } }{ \varphi  }
    \\
    =
    &
    \action{ \nabla\Ecal\sts{v_n\, ; \,\Omega_n } }{ \varphi  }
    \\
    =
    &
    \int_{\R^N}\int_{\R^N} \frac{ \sts{ v_n\sts{x}-v_n\sts{y} }\sts{ \varphi\sts{x}-\varphi\sts{y} } }{ \abs{x-y}^{N+2s} }\dx\dy
    -
    \int_{\Omega_n}\abs{ v_n\sts{x} }^{ 2^{\ast}_{s}-2 }v_n\sts{x} \varphi\sts{x}\dx
    \\
    =
    &
    \int_{\R^N}\int_{\R^N} \frac{ \left[{ u_n\sts{x}-u_n\sts{y} }\right]
    \left[{
        \lambda_n^{ -\frac{N-2s}{2} }\varphi\sts{\frac{x-\xi_n}{\lambda_n} }
        -
        \lambda_n^{ -\frac{N-2s}{2} }\varphi\sts{\frac{y-\xi_n}{\lambda_n} } }
    \right]
    }{ \abs{x-y}^{N+2s} }\dx\dy
    \\
    & -
    \int_{\Omega}\abs{ u_n\sts{x} }^{ 2^{\ast}_{s}-2 }u_n\sts{x} \sts{ \lambda_n^{ -\frac{N-2s}{2} }\varphi\sts{\frac{x-\xi_n}{\lambda_n} } }\dx
    \\
    =
    &
    \action{ \nabla\Ecal\sts{u_n\, ; \,\Omega } }{ \lambda_n^{ -\frac{N-2s}{2} }\varphi\sts{\frac{\cdot-\xi_n}{\lambda_n} } }
    \\
    =
    &
    \so{1}\norm{ \lambda_n^{ -\frac{N-2s}{2} }\varphi\sts{\frac{\cdot-\xi_n}{\lambda_n} } }
    \\
    =
    &
    \so{1}\norm{ \varphi }.
  \end{align*}

  On the other hand, for each $ \varphi\in C_{c}^{\infty}\sts{\R^N}^{\sigma}_{G} $, one has
\begin{equation}
\label{eq:021}
  \iint_{\R^N\times \R^N }\frac{ \sts{ v_n\sts{x}-v_n\sts{y} }\sts{ \varphi\sts{x}-\varphi\sts{y} } }{ \abs{x-y}^{N+2s} }\dx\dy
  \rightarrow
  \iint_{\R^N\times \R^N} \frac{ \sts{ V\sts{x}-V\sts{y} }\sts{ \varphi\sts{x}-\varphi\sts{y} } }{ \abs{x-y}^{N+2s} }\dx\dy,
\end{equation}
moreover,  by \cite[Theorem A.2]{Willem1997} and \eqref{eq:020} with $p=2^{\ast}_{s}-1$, we can  obtain that
\begin{equation*}
    \abs{v_n}^{ 2^{\ast}_{s}-2 }v_n  \rightarrow  \abs{V}^{ 2^{\ast}_{s}-2 }V \quad \text{ in }  L^1_{\mathrm{loc}}\sts{\R^N},
\end{equation*}
which implies that
\begin{equation*}
  \int_{\R^N}\abs{v_n\sts{x}}^{ 2^{\ast}_{s}-2 }v_n\sts{x}\varphi\sts{x}\dx
  \rightarrow
  \int_{\R^N}\abs{V\sts{x}}^{ 2^{\ast}_{s}-2 }V\sts{x}\varphi\sts{x}\dx,
  \text{~~for all~~}\varphi\in C_{c}^{\infty}\sts{\R^N}^{\sigma}_{G}.
\end{equation*}
Therefore
\begin{equation*}
  \action{ \nabla\Ecal\sts{V\, ; \,\R^N } }{ \varphi  }=0,\text{~~for all~~}\varphi\in C_{c}^{\infty}\sts{\R^N}^{\sigma}_{G},
\end{equation*}
hence
\begin{equation*}
  \action{ \nabla\Ecal\sts{V\, ; \,\R^N } }{ \psi  }=0,\text{~~for all~~}\psi\in \dot{H}^s\sts{\R^N}_{G}^{\sigma}.
\end{equation*}
and
\begin{equation*}
  {\Ncal}\sts{V\, ; \,\R^N }=0.
\end{equation*}
So $V\in \Ncal(\R^{N})_{G}^{\sigma}$, and
$$ \norm{v_n - V}\rightarrow 0. $$


  \item[\bf Case II.] $ \frac{ 1 }{ \lambda_n }\lim\limits_{n\to\infty}\dist{ \xi_n }{\partial\Omega}=\rho $ with
  $\rho\in\sts{0,\infty}$.

When $\left\{ \xi_n \right\}_{n=1}^{\infty}\subseteq \overline{\Omega}$, let
    \begin{equation*}
      \Hcal
     :=
      \setcdt{ x\in\R^N }{ x\cdot\nu>-\rho },
    \end{equation*}
    where $\nu$ is the inward pointing unit normal to $\partial\Omega$ at $\xi^{0}$. Since $\xi^{0}\in\sts{\R^N}^G$, so is $\nu$. Thus, we have $ \Hcal$ is $G$-invariant.

When $\left\{ \xi_n \right\}_{n=1}^{\infty}\nsubseteq \overline{\Omega}$, let
    \begin{equation*}
      \Hcal
    :=
      \setcdt{ x\in\R^N }{ x\cdot\nu>\rho }.
    \end{equation*}
    Similarly, we also have $ \Hcal$ is $G$-invariant.

    In both cases, it is easy to see that $\Hcal$ is a half plane in  $\R^{N}$, therefore if $X$ is compact and $X\subseteq \Hcal$, there exists $n_{0}$ such that $X\subseteq\Omega_{n}$ for all $n\geq n_{0}$, moreover if $X$ is compact and $X\subseteq \R^{N}\setminus \Hcal$, then $X\subseteq \R^{N}\setminus \Omega_{n}$ for $n$ large enough. As $V_{n}\rightarrow V \  \text{a.e.}$ in $\R^{N}$, this implies $V\neq 0$ in $\Omega_{n}$, in particular, $V=0\  \text{a.e.}$ in $\R^{N}\setminus \Hcal$. So $V\in \dot{H}^s\sts{\Hcal}^{\sigma}_{G}$. Moreover, for each $\varphi \in C_{c}^{\infty}(\Hcal)_{G}^{\sigma}$, for $n$ sufficiently large, we have
    \begin{align*}
    &
    \action{ \nabla\Ecal\sts{v_n\, ; \,\Hcal } }{ \varphi  }
    \\
    =
    &
    \action{ \nabla\Ecal\sts{v_n\, ; \,\Omega_n } }{ \varphi  }
    \\
    =
    &
    \int_{\R^N}\int_{\R^N} \frac{ \sts{ v_n\sts{x}-v_n\sts{y} }\sts{ \varphi\sts{x}-\varphi\sts{y} } }{ \abs{x-y}^{N+2s} }\dx\dy
    -
    \int_{\Omega_n}\abs{ v_n\sts{x} }^{ 2^{\ast}_{s}-2 }v_n\sts{x} \varphi\sts{x}\dx
    \\
    =
    &
    \int_{\R^N}\int_{\R^N} \frac{ \left[{ u_n\sts{x}-u_n\sts{y} }\right]
    \left[{
        \lambda_n^{ -\frac{N-2s}{2} }\varphi\sts{\frac{x-\xi_n}{\lambda_n} }
        -
        \lambda_n^{ -\frac{N-2s}{2} }\varphi\sts{\frac{y-\xi_n}{\lambda_n} } }
    \right]
    }{ \abs{x-y}^{N+2s} }\dx\dy
    \\
    & -
    \int_{\Omega}\abs{ u_n\sts{x} }^{ 2^{\ast}_{s}-2 }u_n\sts{x} \sts{ \lambda_n^{ -\frac{N-2s}{2} }\varphi\sts{\frac{x-\xi_n}{\lambda_n} } }\dx
    \\
    =
    &
    \action{ \nabla\Ecal\sts{u_n\, ; \,\Omega } }{ \lambda_n^{ -\frac{N-2s}{2} }\varphi\sts{\frac{\cdot-\xi_n}{\lambda_n} } }
    \\
    =
    &
    \so{1}\norm{ \lambda_n^{ -\frac{N-2s}{2} }\varphi\sts{\frac{\cdot-\xi_n}{\lambda_n} } }
    \\
    =
    &
    \so{1}\norm{ \varphi }.
  \end{align*}
   On the other hand, for each $ \varphi\in C_{c}^{\infty}\sts{\Hcal}^{\sigma}_{G} $, one has
\begin{equation}
\label{eq:041}
 \iint_{\R^N\times \R^N} \frac{ \sts{ v_n\sts{x}-v_n\sts{y} }\sts{ \varphi\sts{x}-\varphi\sts{y} } }{ \abs{x-y}^{N+2s} }\dx\dy
  \rightarrow
  \iint_{\R^N\times \R^N} \frac{ \sts{ V\sts{x}-V\sts{y} }\sts{ \varphi\sts{x}-\varphi\sts{y} } }{ \abs{x-y}^{N+2s} }\dx\dy,
\end{equation}
and by \cite[Theorem A.2]{Willem1997}, we can see that
\begin{equation*}
    \abs{v_n}^{ 2^{\ast}_{s}-2 }v_n  \rightarrow  \abs{V}^{ 2^{\ast}_{s}-2 }V \quad \text{ in }  L^1_{\mathrm{loc}}\sts{\Hcal},
\end{equation*}
which implies that
\begin{equation*}
  \int_{\Hcal}\abs{v_n\sts{x}}^{ 2^{\ast}_{s}-2 }v_n\sts{x}\varphi\sts{x}\dx
  \rightarrow
  \int_{\Hcal}\abs{V\sts{x}}^{ 2^{\ast}_{s}-2 }V\sts{x}\varphi\sts{x}\dx,
  \text{~~for all~~}\varphi\in C_{c}^{\infty}\sts{\Hcal}^{\sigma}_{G}.
\end{equation*}
Therefore
\begin{equation*}
  \action{ \nabla\Ecal\sts{V\, ; \,\Hcal } }{ \varphi  }=0,\text{~~for all~~}\varphi\in C_{c}^{\infty}\sts{\Hcal}^{\sigma}_{G},
\end{equation*}
hence
\begin{equation*}
  \action{ \nabla\Ecal\sts{V\, ; \,\Hcal } }{ \psi  }=0,\text{~~for all~~}\psi\in \dot{H}^s\sts{\Hcal}_{G}^{\sigma}.
\end{equation*}
and
\begin{equation*}
  {\Ncal}\sts{V\, ; \,\Hcal }=0.
\end{equation*}
So $V\in \Ncal(\Hcal)_{G}^{\sigma}$,
we obtain that $$ \norm{v_n - V}\rightarrow 0. $$

After a change of variable, we have
\begin{equation}
\begin{aligned}
\label{eq:042}
 &\parallel u_{n}(x)-\lambda_{n}^{\frac{-(N-2s)}{2}}V(\frac{x-\xi_{n}}{\lambda_{n}})\parallel^{2}\\
 &=\iint_{\R^{N}\times \R^N}\frac{\mid u_{n}(x)-u_{n}(y)-\lambda_{n}^{\frac{-(N-2s)}{2}}V(\frac{x-\xi_{n}}{\lambda_{n}})+\lambda_{n}^{\frac{-(N-2s)}{2}}V(\frac{y-\xi_{n}}{\lambda_{n}})\mid^{2}}{\mid x-y\mid^{N+2s}}\dx\dy\\
 &=\iint_{\R^{N}\times \R^N}\frac{\mid \lambda_{n}^{\frac{N-2s}{2}}u_{n}(\lambda_{n}x+\xi_{n})-\lambda_{n}^{\frac{N-2s}{2}}u_{n}(\lambda_{n}y+\xi_{n})-V(x)+V(y)\mid^{2}}{\mid x-y\mid^{N+2s}}\dx\dy\\
 &=\iint_{\R^{N}\times \R^N}\frac{\mid v_{n}(x)-v_{n}(y)-V(x)+V(y)\mid^{2}}{\mid x-y\mid^{N+2s}}\dx\dy\\
 &=\parallel v_{n}-V \parallel^{2},
\end{aligned}
\end{equation}
therefore we obtain that
 \begin{align*}
      \lim_{n\to\infty}\norm{ u_n\sts{\cdot}-\frac{ 1 }{ \lambda_n^{\frac{ N-2s }{ 2 }} }V\sts{ \frac{ \cdot-\xi_{n} }{ \lambda_n } } }=0,
    \end{align*}
  \end{enumerate}
 which completes the proof.
\end{proof}

\section{Entire nodal solutions}
\label{sec:mainresult}
In this section we prove the main result.
\begin{lemm}
\label{lem:term}Let $G$ be a closed subgroup of $O(N)$ and $\sigma : G\rightarrow Z_{2}$ be a continuous homomorphism  and is surjective, and the group $G$ satisfies \ref{assum:G:orbit}-\ref{assum:G:null}. Then $\Ecal\sts{ u \, ; \,\R^{N}}$ attains its minimum on $\Nscr\sts{\R^{N}}^{\sigma}_{G}$. Consequently, the problem \eqref{nls} has a nontrival $\sigma$-equivariant solution in  $  \dot{H}^s\sts{\R^N}$.
\end{lemm}
\begin{proof}
The unit ball $\Omega := \{x\in \R^{N}: \mid x\mid<1\}$ is $G$-invariant for every $G$, as $0\in\Omega$, we have that $\Omega^{G}\neq \emptyset$. Then, by \Cref{lem:hari}, we know $\mathrm{m}\sts{\Omega}_{ G }^{\sigma} = \mathrm{m}\sts{\R^N}_{ G }^{\sigma}.$

By \Cref{hari:cd}, there exists a subsequence $\{u_{n}\}_{n=1}^{\infty}$ satisfying
\begin{align*}
   \Ecal\sts{ u_n  \, ; \,\Omega }\rightarrow \mathrm{m}\sts{\Omega}_{ G }^{\sigma},
   \text{ and }
   \nabla\Ecal\sts{ u_n  \, ; \,\Omega }\rightarrow 0 \text{  in  } \sts{ \dot{H}^s\sts{\Omega}_{G}^{\sigma} }{'}
   \text{ as } n\rightarrow \infty.
\end{align*}
Then \Cref{lem:ps} asserts that there are two possibilities :
\begin{enumerate}[label=\textbf{Case } \textbf{\roman*}., ref=\textbf{Case } \textbf{\roman*.},]

\item\label{lemm:nehair5} there exists $u\in \Nscr\sts{\Omega}^{\sigma}_{G}$, such that $\Ecal\sts{ u \, ; \,\Omega} =\mathrm{m}\sts{\Omega}_{ G }^{\sigma}.$

\item\label{lemm:nehair6} there exists $V\in \Nscr\sts{\Hcal}^{\sigma}_{G}$, such that $\Ecal\sts{ V \, ; \,\Hcal} =\mathrm{m}\sts{\Hcal}_{ G }^{\sigma}. $
\end{enumerate}

As $\Nscr\sts{\Theta}_{G}^{\sigma}\subset \Nscr\sts{\R^{N}}_{G}^{\sigma}$ for every $G$-invariant domain $\Theta$ in $\R^{N}$, from the above two possibilities we can get $\Ecal\sts{\cdot \, ; \,\R^{N}} $ attains its minimum on $\Nscr\sts{\R^{N}}_{G}^{\sigma}.$
\end{proof}

In order to prove Theorem \ref{thm:main},  it suffices to show that there are $n$ groups with the properties stated in the following lemma.
 \begin{lemm}
 \label{lem:neharic}
  Let $N=4n+m$ with $n\geq1$ and $m\in\{0,1,2,3\}$, then for each $j=1,\cdot\cdot\cdot,n$, there exists a closed subgroup $G_{j}$ of $O(N)$ and a continuous homomorphism $\sigma_{j} : G_{j}\rightarrow Z_{2}$ with the followinng properties :
 \begin{enumerate}[label={(}\texttt{\alph{*}}{)}, ref={(}{\alph{*}}{)}]
 \item\label{lemm:hari2h}$G_{j}$ and $\sigma_{j}$ satisfy \ref{assum:G:orbit}-\ref{assum:G:null}.
 \item\label{lemm:hari3h}If $u,v : \R^{N}\rightarrow \R$ are nontrivial functions, $u$ is $\sigma_{i}$-equivariant and $v$ is $\sigma_{j}$-equivariant with $i\neq j$, then $u\neq v$.
 \end{enumerate}
 \end{lemm}
 \begin{proof}
 Let $\Gamma$ be the group generated by $\{e^{i\theta}, \varrho : \theta \in[0,2\pi)\}$, acting on $\C^{2}$ by
 $$ e^{i\theta}(\zeta_{1},\zeta_{2}):= (e^{i\theta}\zeta_{1},e^{i\theta}\zeta_{2}),\quad \varrho(\zeta_{1},\zeta_{2}):= (-\overline{\zeta}_{2},\overline{\zeta}_{1})\quad\text{for}\ (\zeta_{1},\zeta_{2})\in \C^{2}, $$
 and let $\sigma: \Gamma\rightarrow\Z_{2}$ be the homomorphism given by $\sigma(e^{i\theta}):=1$ and $\sigma(\varrho):= -1$. Note that the $\Gamma$-orbit of a point $z\in \C^{2}$ is the union of two circles that lie in orthogonal planes if $z\neq 0$, and it is ${0}$ if $z=0$.

 Define $\Lambda_{j}:=O(N-4j)$ if $j=1,\cdot\cdot\cdot,n-1$ and $\Lambda_{n}:= 1$, then $\Lambda_{j}$-orbit of a point $y\in \R^{N-4j}$ is an $(N-4j-1)$-dimensional sphere if $j=1,\cdot\cdot\cdot,n$, and it is a single point if $j=n$.

 Define $G_{j}=\Gamma_{j}\times\Lambda_{j}$, acting coordinatewise on $\R^{N}\equiv (\C^{2})^{j}\times \R^{N-4j}$, i.e.,
 \begin{equation}
 (\gamma_{1},\cdot\cdot\cdot,\gamma_{j},\eta)\sts{z_{1},\cdot\cdot\cdot,z_{j},y}^{\intercal}=\sts{\gamma_{1}z_{1},\cdot\cdot\cdot,\gamma_{j}z_{j},\eta y}^{\intercal},
 \end{equation}
where $\gamma_{i}\in\Gamma$, $\eta\in\Lambda_{j}$, $z_{i}\in\C^{2}$ and $y\in \R^{N-4j}$, and $\sigma_{j}: G_{j}\rightarrow\Z_{2}$ be the homomorphism
\begin{equation}
\sigma_{j}(\gamma_{1},\cdot\cdot\cdot,\gamma_{j},\eta)=\sigma(\gamma_{1})\sigma(\gamma_{2})\cdots\sigma(\gamma_{j}).
\end{equation}
Obviously $\sigma_{j}$ is surjective.

 We firstly show that  $G_{j}$ and $\sigma_{j}$ satisfy \ref{assum:G:orbit}-\ref{assum:G:null}. On one hand, let $x=(x_{1},\cdot\cdot\cdot,x_{j},y)\in \R^{N}$, and if $G_{j}x\neq  \{x\}$, there exists $(\gamma_{1},\cdot\cdot\cdot,\gamma_{j},\eta)$ and $(\beta_{1},\cdot\cdot\cdot,\beta_{j},\alpha)\in G_{j}$ such that
\begin{equation}
(\gamma_{1},\gamma_{2},\cdot\cdot\cdot,\gamma_{j},\eta)x\neq(\beta_{1},\beta_{2},\cdot\cdot\cdot,\beta_{j},\alpha)x\in G_{j}{ \cdot }x,
\end{equation}
therefore $\dim(G_{j} { \cdot } x)>0$ and
\begin{align*}
\dim\sts{G_{j}  { \cdot} x}=
  \begin{cases}
   N-2j-1 , & \mbox{if } j=1,\dots,n-1, \\
   2j+1 , & \mbox{if } j=n.
  \end{cases}
\end{align*}

On the other hand, let $\xi=(\xi_{1},\cdot\cdot\cdot,\xi_{j},0)$, suppose $g\xi=\xi$, it is easy to see
$$g {=} \left(\e^{\i k_{1}\theta}\varrho^{4n_{1}},\e^{\i k_{2}\theta}\varrho^{4n_{2}},\cdot\cdot\cdot,e^{\i k_{j}\theta}\varrho^{4n_{j}},1\right),$$
where $k_{j}\in\mathbb{N}$, $n_{j}\in\mathbb{N}$ and $\theta=0$.

Therefore, we obtain that
\begin{align*}
\sigma_{j}(g) =\sigma(\e^{\i k_{1}\theta}\varrho^{4n_{1}})\sigma(\e^{\i k_{2}\theta}\varrho^{4n_{2}})\cdots\sigma(\e^{\i k_{j}\theta}\varrho^{4n_{j}})
  =1.
\end{align*}
Therefore, \ref{assum:G:null} hold.

Secondly, we prove \ref{lemm:hari3h}. Suppose $u\neq 0$ is $\sigma_{i}$-equivariant and 
$v\neq 0$ is $\sigma_{j}$-equivariant with $i<j$, 
and $u(x)=v(x)$ for some $x=(z_{1},\cdot\cdot\cdot,z_{j},y)\in(\C^{2})^{j}\times \R^{N-4j}$. Then
\begin{align*}
&
u(z_{1},\cdot\cdot\cdot,\varrho_{j}z_{j},y)\\
=&u(z_{1},\cdot\cdot\cdot,z_{j},y)\\
=&v(z_{1},\cdot\cdot\cdot,z_{j},y)\\
=&-v(z_{1},\cdot\cdot\cdot,\varrho_{j}z_{j},y),
\end{align*}
which implies that $u(z_{1},\cdot\cdot\cdot,\varrho_{j}z_{j},y)\neq v(z_{1},\cdot\cdot\cdot,\varrho_{j}z_{j},y)$.
Therefore $u\neq v$.
 \end{proof}

We are ready to prove Theorem \ref{thm:main}.
\begin{proof}[Proof of \Cref{thm:main}]
Let $N=4n+m$ with $n\geq1$ and $m\in\{0,\cdot\cdot\cdot,3\}$, for each $j=1,\cdot\cdot\cdot,n$, let $G_{j}$ be the closed subgroup of $O(N)$ and $\sigma_{j}$ be the continuous homomorphisn given by \Cref{lem:neharic}. The \Cref{lem:term} tell us that there  be $\sigma_{j}$-equivariant solution $u_{j}$ to the problem \eqref{nls}. \Cref{lem:neharic} shows that the solutions $u_{1},\cdot\cdot\cdot,u_{n}$ are pairwise distinct.
\end{proof}

\end{document}